\documentclass[12pt,amsymb,fullpage]{amsart}
\usepackage{amssymb,amscd,pstricks}

\newtheorem{theorem}{Theorem}[section]
\newtheorem{defn}[theorem]{Definition}

\newtheorem{lemma}[theorem]{Lemma}

\newtheorem{eple}[theorem]{Example}
\newtheorem{rmk}[theorem]{Remarks}
\newtheorem{dsc}[theorem]{Discussion}
\newtheorem{nota}[theorem]{Notation}

\newsavebox{\indbin}
\savebox{\indbin}{\begin{picture}(0,0)
\newlength{\gnu}
\settowidth{\gnu}{$\smile$} \setlength{\unitlength}{.5\gnu}
\put(-1,-.65){$\smile$} \put(-.25,.1){$|$}
\end{picture}}

\newcommand{\be}{\begin{enumerate}}
\newcommand{\bd}{\begin{defn}}
\newcommand{\bt}{\begin{theorem}}
\newcommand{\bl}{\begin{lemma}}
\newcommand{\ee}{\end{enumerate}}
\newcommand{\ed}{\end{defn}}
\newcommand{\et}{\end{theorem}}
\newcommand{\el}{\end{lemma}}

\begin{document}
\title{Nonstandard Methods for Solving the Heat Equation}
\author{Tristram de Piro}
\address{Flat 3, Redesdale House, 85 The park, Cheltenham, GL50 2RP }
 \email{t.depiro@curvalinea.net}
\thanks{}
\begin{abstract}
We apply convergence results for discrete Markov chains, to prove the existence of an equilibrium limit in the nonstandard heat equation. We construct a nonstandard backward martingale from a nonstandard solution, and show, using the Feynman-Kac method, how to derive an explicit formula for such solutions, when the initial condition is S-continuous. Finally, we prove that that the nonstandard solution to the heat equation, with a smooth initial condition, specialises to the classical solution.
\end{abstract}
\maketitle

This paper is concerned with nonstandard approaches to the heat equation. Arguably, interest in these methods goes back to Joseph Fourier, (1768-1830), and Pierre Simon Laplace, (1749-1827), who preferred the use of Newtonian infinitesimals, before a standard version of the calculus was available, in about 1820. Indeed, Fourier wrote an essay, "Theorie du mouvement de la chaleur dans les corps solides", in 1811, published between 1819 and 1820, in which he considers the solution of the heat equation on an infinite line, obtaining an explicit solution with the use of Fourier transforms. In an earlier work of 1807, "Memoire sur la propogation de la chaleur", he employs a Fourier series solution which Laplace later recognised as solving the heat equation on a bounded domain.\\

Laplace's work on probability in connection with the heat equation is also interesting. In 1809, in his "Memoires sur les Approximations des Formules qui sont Fonctions de Tres Grandes Nombres et sur leur Application aux Probabilites", Laplace derives the Central Limit Theorem. In his later 1814 essay, "Essai Philosophique sur les Probabilites des Jeux", he formulates the idea behind martingale strategies for fair games, a precursor to modern nonstandard stochastic analysis. All this is tied in with work on the heat equation, using the method of finding probability distributions by differential equations. He does this to find the distribution of the average inclination of n independent satellite orbits in his 1809 memoir.\\

Fourier's method now constitutes a core of modern analysis, and we consider this technique in the final part of the paper. Laplace's work anticipates a general probabilistic method referred to as the Fokker-Planck formula, the converse method, using stochastic processes as a way of solving the heat equation, is generally now known as the Feynman-Kac formula. This converse method, in the guise of reverse martingales, constitutes the second part of this paper.\\

The results of this paper are particularly interesting because they clarify work due to Fourier and Laplace which has been lost. They are also very relevant to modern mathematics, which has seen an explosion of interest in nonstandard stochastic analysis over the past 40 years. The methods of the last two parts of this paper can be applied to other partial differential equations, most importantly Schrodinger's equation, which with the insertion of an appropriate constant, is identical to the heat equation in form. Indeed, it is the author's hope that this paper can serve as a template for analysis of the Schrodinger equation for a free particle. A statistical interpretation of the propagator for such an equation would greatly simplify work on the Feynman path integral by replacing the totality of paths with the paths of Brownian motion. It is known that such standard methods lead to insights into certain partial differential equations with a potential term, via the Feynman-Kac formula. The connection with Schrodinger's equation with a potential is a new possibility, and one would envisage that the techniques of nonstandard stochastic analysis, see \cite{and},\cite{alb} and \cite{adv}, Chapter 9, on the Martingale Representation Theorem, might become very relevant here. Some preliminary work on a nonstandard solution to the heat equation was done in \cite{han}, but for an unbounded domain, while we consider a bounded domain. The use of statistical methods is mainly avoided there, with an application of Stirling's formula rather than the Central Limit Theorem to achieve the final result in IV.13.  \\

The paper is divided into three parts. We identify the principal results. In Theorem \ref{measurable}, we establish a rate of convergence to equilibrium, using the theory of discrete Markov chains. This ends the first part of the paper. In Lemma \ref{equivalent}, we establish the statistical nature of the heat equation. We apply the method of reverse martingales in Theorem \ref{martingale}. After some error estimates for the Central Limit Theorem, in Lemma \ref{centrallimit}, we achieve the main result of the second part which is an explicit description of the nonstandard solution to the heat equation, Theorem \ref{solution}. On specialisation, this agrees with classical result, Theorem \ref{convolution}, which will be further verified in the final part of the paper. The final part employs the nonstandard theory of Fourier analysis, to find a different approach, which corresponds to the classical quantum theory. The main result here is Theorem \ref{specialises}, which shows that the nonstandard solution specialises to the classical solution, in the case where we start with a smooth initial condition. We finally note, in Theorem \ref{equilibrium}, that our final two approaches provide a much faster convergence to equilibrium than given by the Markov theory presented here. However, a faster rate of convergence, in the Markov setting, applying results in discrete harmonic analysis, can be found in \cite{Sil1}, see also \cite{Sil2}, although a slightly different Markov chain is used there. The former book might provide insights into an explicit solution when the initial condition fails to be bounded or $S$-continuous.\\

One of the most fundamental results in the theory of Markov chains is the following;\\

\begin{theorem}
\label{coupling}
Let $P$ be the transition matrix of an irreducible, aperiodic,positive recurrent Markov chain, $\{X_{n}\}_{n\geq 0}$, with invariant distribution $\pi$. Then, for any initial distribution, $P(X_{n}=j)\rightarrow\pi_{j}$, as $n\rightarrow\infty$. In particular;\\

$p_{ij}^{(n)}\rightarrow \pi_{j}$, for all states $i,j$, as $n\rightarrow\infty$\\

\end{theorem}

\begin{proof}
A good reference for this result is \cite{N}. However, we give the proof as it is used and modified later. Let the initial distribution be $\lambda$, and let $I$ be the state space. Choose $\{Y_{n}\}_{n\geq 0}$, such that $\{X_{n}\}_{n\geq 0}$ and $\{Y_{n}\}_{n\geq 0}$ are independent, with $\{Y_{n}\}_{n\geq 0}$ Markov $(\pi,P)$. Let $T=inf\{n\geq 1:X_{n}=Y_{n}\}$. We claim that $P(T<\infty)=1$, $(*)$. Let $W_{n}=(X_{n},Y_{n})$. Then $\{W_{n}\}_{n\geq 0}$ is a Markov chain on $I\times I$. By independence, it has transition probabilities given by;\\

$\overline{p}_{(i,j)(k,l)}=p_{ik}p_{jl}$ $(\dag)$\\

and initial distribution $\mu_{(i,j)}=\lambda_{i}\pi_{j}$. A simple calculation, using $(\dag)$, shows that;\\

$\overline{p}^{(n)}_{(i,j)(k,l)}=p_{ik}^{(n)}p_{jl}^{(n)}$ for fixed states $i,j,k,l$\\

As $P$ is irreducible and aperiodic, we have that $min(p_{ik}^{(n)},p_{jl}^{(n)})>0$, for sufficiently large $n$. Hence, for such $n$, $\overline{p}^{(n)}_{(i,j)(k,l)}>0$ and $\overline{P}$ is irreducible. A similar straightforward calculation gives that the distribution $\pi_{(i,j)}=\pi_{i}\pi_{j}$ is invariant for $\overline{P}$. By well known results, this implies that $\overline{P}$ is positive recurrent. Fix a state $b\in I$, and let $S=inf\{n\geq 1:X_{n}=Y_{n}=b\}$. Then $S$ is the first passage time in the system $\{W_{n}\}_{n\geq 0}$ to $(b,b)$, and $P(S<\infty)=1$ follows by known results, and the fact that $\overline{P}$ is irreducible and recurrent. Clearly $P(S<\infty)\leq P(T<\infty)$, so $(*)$ follows. We now calculate;\\

$P(X_{n}=j)=P(X_{n}=j,n\geq T)+P(X_{n}=j,n<T)$\\

$=P(Y_{n}=j,n\geq T)+P(X_{n}=j,n<T)$\\

by definition of $T$ and the fact that $\{X_{n}\}_{n\geq 0}$ and $\{Y_{n}\}_{n\geq 0}$ have the same transition matrix. Then;\\

$P(X_{n}=j)=P(Y_{n}=j,n\geq T)+P(Y_{n}=j,n<T)$\\

$\indent \ \ \ \ \ \ \ \ \ \ \ -P(Y_{n}=j,n<T)+P(X_{n}=j,n<T)$\\

$=P(Y_{n}=j)-P(Y_{n}=j,n<T)+P(X_{n}=j,n<T)$\\

$=\pi_{j}-P(Y_{n}=j,n<T)+P(X_{n}=j,n<T)$ $(**)$\\

We have that $P(Y_{n}=j,n<T)\leq P(n<T)$ and $P(n<T)\rightarrow P(T=\infty)=0$ as $n\rightarrow\infty$, using $(*)$. Similarly, $P(X_{n}=j,n<T)\rightarrow 0$ as $n\rightarrow\infty$. It follows that $P(X_{n}=j)\rightarrow \pi_{j}$, using $(**)$, as required. The final claim is a consequence of the fact that $p_{ij}^{(n)}=P(X_{n}=j)$ where the initial distribution of $X_{0}$ is the dirac function $\delta_{i}$.

\end{proof}

We now establish a rate of convergence result.

\begin{lemma}
\label{convergence}
Let $P$ be the transition matrix for a finite irreducible aperiodic Markov chain. Then there exists $m\geq 1$ and $\rho\in (0,1)$, such that;\\

$|p_{ij}^{(n)}-\pi_{j}|\leq (1-\rho)^{{n\over m}-1}$, for all states $i,j$\\

where $\pi$ is the limiting distribution guaranteed by Theorem \ref{coupling}.

\end{lemma}

\begin{proof}
From Theorem \ref{coupling}, taking the initial distribution of $X_{0}$ to be $\delta_{i}$, we have that;\\

$P(X_{n}=j)=\pi_{j}-P(Y_{n}=j,n<T)+P(X_{n}=j,n<T)$\\

Hence;\\

$|p_{ij}^{(n)}-\pi_{j}|\leq P(n<T)$\\

As $P$ is irreducible and aperiodic, we have that $p_{kl}^{(n)}>0$ for all sufficiently large $n$, and all states $k,l$. As $P$ is finite, there exists an $m\geq 1$ such that $p_{kl}^{(m)}>0$ for all $k,l$. In particular, there exists $\rho\in (0,1)$ such that $p_{kl}^{(m)}\geq\rho$. We have that;\\

$P_{k,l}(T\leq m)\geq \sum_{u} \overline{p}^{(m)}_{(k,l),(u,u)}=\sum_{u}p_{ku}^{(m)}p_{lu}^{(m)}\geq \rho\sum_{u}p_{ku}^{(m)}=\rho$\\

$P_{k,l}(T>m)\leq (1-\rho)$\\

$P(T>m)=\sum_{(k,l)}P_{(k,l)}(T>m)\delta_{ik}\pi_{l}\leq (1-\rho)$\\

Moreover;\\

$P(T>n)\leq P(T>[{n\over m}]m)$\\

We claim that, for $k\geq 1$, $P(T>(k+1)m|T>km)\leq 1-\rho$. We have, using the total law of probability, the Markov property and the definition of $T$, that;\\

$P(T>(k+1)m|T>km)$\\

$=\sum_{i_{km}\neq j_{km}, i_{km-1}\neq j_{km-1},\ldots,i_{0}\neq j_{0}}P(T>(k+1)m|W_{km}=(i_{km},j_{km}),W_{km-1}=(i_{km-1},j_{km-1}),\ldots, W_{0}=(i_{0},j_{0}))P(W_{km}=(i_{km},j_{km}),W_{km-1}=(i_{km-1},j_{km-1}),\ldots, W_{0}=(i_{0},j_{0})|T>km)$\\

$=\sum_{i_{km}\neq j_{km}, i_{km-1}\neq j_{km-1},\ldots,i_{0}\neq j_{0}}P(T>(k+1)m|W_{km}=(i_{km},j_{km}))P(W_{km}=(i_{km},j_{km}),W_{km-1}=(i_{km-1},j_{km-1}),\ldots, W_{0}=(i_{0},j_{0})|T>km)$\\

$\leq (1-\rho)\sum_{i_{km}\neq j_{km}, i_{km-1}\neq j_{km-1},\ldots,i_{0}\neq j_{0}}P(W_{km}=(i_{km},j_{km}),W_{km-1}=(i_{km-1},j_{km-1}),\ldots, W_{0}=(i_{0},j_{0})|T>km)$\\

$=(1-\rho)$\\

Inductively, we have that;\\

$P(T>km)=P(T>km,T>(k-1)m)$\\

$=P(T>km|T>(k-1)m)P(T>(k-1)m)$\\

$\leq P(T>km|T>(k-1)m)(1-\rho)^{k-1}$\\

$\leq (1-\rho)^{k}$\\

It follows that $|p_{ij}^{(n)}-\pi_{j}|\leq (1-\rho)^{[{n\over m}]}\leq (1-\rho)^{{n\over m}-1}$, as required. \\

\end{proof}

\begin{lemma}
\label{transition}
Let $P$ define a Markov chain with $N$ states, $\{0,1,\ldots,N-1\}$, where $N$ is odd,  such that the transition probabilities of moving from state i to i-2,i,i+2 (mod N) respectively are ${1\over 2}$. Then $P$ is irreducible, aperiodic and $\pi$, defined by $\pi_{i}={1\over N}$, for $0\leq i\leq N-1$, defines an invariant distribution. Moreover, we can choose $m=2N$ and $\rho={1\over 4^{N}}$ in Lemma \ref{convergence}. It follows that;\\

$|p_{ij}^{(n)}-{1\over N}|\leq ({4^{N}-1\over 4^{N}})^{{n\over 2N}-1}=\epsilon_{n}$\\

Moreover, for any initial probability distribution $\lambda_{0}$, letting $\lambda_{j}^{n}=P(X_{n}=j)$, we have that;\\

$|\lambda_{j}^{n}-{1\over N}|\leq \epsilon_{n}$, $0\leq j\leq N-1$ $(*)$\\

For any initial distribution $\mu_{0}=\mu_{0}^{+}-\mu_{0}^{-}$, with sums $K^{+}$ and $K^{-}$, letting $\mu^{n}=\mu_{0}P^{n}$, and $K=K^{+}-K^{-}$ we have that;\\

$|\mu_{j}^{n}-{K\over N}|\leq (K^{+}+K^{-})\epsilon_{n}$, $0\leq j\leq N-1$\\

\end{lemma}

\begin{proof}
To prove irreducibility, observe that $i+2({N+1\over 2})=i+1$ (mod $N$), hence, $p_{i,i+1}^{({N+1\over 2})}\geq {1\over 2^{{N+1\over 2}}}$, for all states $0\leq i\leq N-1$, $(*)$. To show that all states $i,j$ communicate, it is sufficient, by symmetry, to assume that $i\leq j$. If $i=j$, then we have that $p_{i,i}^{(2)}\geq {1\over 4}$. If $j-i$ is even, we have that $p_{i,j}^{({j-i\over 2})}\geq {1\over 2^{{j-i\over 2}}}$. If $j-i$ is odd, then $j-(i+1)$ is even. We then have that;\\

 $p_{i,j}^{{N+1\over 2}+{j-(i+1)\over 2}}\geq {1\over 2^{{N+1\over 2}}}{1\over 2^{{j-(i+1)\over 2}}}$\\

using $(*)$. To prove aperiodicity, it is sufficient to show that $p_{ii}^{(n)}>0$, for sufficiently large $n$, for any given state $i$ with $0\leq i\leq N-1$. Observe that $i+2N=i$ (mod $N$), hence $p_{ii}^{(N)}\geq {1\over 2^{N}}$. If $n\geq N$, and $n$ is even, then clearly $p_{ii}^{(n)}\geq {1\over 2^{n}}$. If $n\geq N$ and $n$ is odd, then $n-N$ is even, and $p_{ii}^{(n)}\geq {1\over 2^{N}}{1\over 2^{n-N}}={1\over 2^{n}}$. For the invariance claim, we compute;\\

$(\pi P)_{i}=\sum \pi_{j}P_{ji}={1\over N}(P_{i-2,i}+P_{i+2,i})={1\over N}({1\over 2}+{1\over 2})={1\over N}$\\

To find $m$ and $\rho$, observe that, by the aperiodicity calculation, that $p_{ii}^{(n)}\geq {1\over 2^{n}}$, for any $0\leq i\leq N-1$, and $n\geq N$. Observe also that, starting at a given state $i$, we can cover all the states, by moving in one direction, a total of $N$ steps. It follows that $p_{ij}^{(k)}\geq {1\over 2^{k}}$, for \emph{some} $k\leq N$. Choosing some $1\leq k_{ij}\leq N$ for each pair of states $(i,j)$, observe that $2N-k_{ij}\geq N$, therefore, for any states $(i,j)$;\\

$p_{ij}^{(2N)}\geq {1\over 2^{k_{ij}}}{1\over 2^{2N-k_{ij}}}={1\over 4^{N}}$\\

We can, therefore, take $m=2N$ and $\rho={1\over 4^{N}}$. We then have that;\\

$(1-\rho)^{{n\over m}-1}=({4^{N}-1\over 4^{N}})^{{n\over 2N}-1}$\\

and the following claim follows, by Lemma \ref{convergence}. The penultimate claim follows by noting that $\lambda_{n}=\lambda_{0}P^{n}$ and calculating;\\

$\lambda_{j}^{n}=\lambda_{0}^{0}p_{0j}^{(n)}+\lambda_{1}^{0}p_{1j}^{(n)}+\ldots+\lambda_{N-1}^{0}p_{N-1,j}^{(n)}$\\

$=(\lambda_{0}^{0}+\ldots+\lambda_{N-1}^{0})({1\over N})+\lambda_{0}^{0}\epsilon_{n}^{0}+\ldots+\lambda_{N-1}^{0}\epsilon_{n}^{N-1}$\\

$={1\over N}+\epsilon_{n}'$\\

where $\epsilon_{n}^{j}\leq \epsilon_{n}$, for $0\leq j\leq N-1$ and $\epsilon_{n}'\leq\epsilon_{n}$. The final claim follows by observing that;\\

 $\mu^{n}=\mu_{0}P^{n}=\mu_{0}^{+}P^{n}-\mu_{0}^{-}P^{n}=K^{+}\pi_{0}^{+}P^{n}-K^{-}\pi_{0}^{-}P^{n}$ $(**)$\\

 where $\{\pi_{0}^{+},\pi_{0}^{-}\}$ are distributions. We then have, using the previous result, and multiplying by an appropriate constant, that;\\

 $|(K^{+}\pi_{0}^{+}P^{n})_{j}-{K^{+}\over N}|\leq K^{+}\epsilon_{n}$\\

 $|(K^{-}\pi_{0}^{-}P^{n})_{j}-{K^{-}\over N}|\leq K^{-}\epsilon_{n}$\\

 for $0\leq j\leq N-1$. Therefore, combining this with $(**)$, we obtain that;\\

 $|\mu_{j}^{n}-{K\over N}|=|\mu_{j}^{n}-({K^{+}-K^{-}\over N})|\leq (K^{+}+K^{-})\epsilon_{n}$\\

 as required.\\

\end{proof}

\begin{lemma}
\label{nonstandard1}
Let $P$ define a non standard Markov chain with $\eta$ states, $\{0,1,\ldots,\eta-1\}$, for $\eta$ odd infinite, such that the transition probabilities of moving from state i to i-2,i+2 (mod $\eta$) respectively are ${1\over 2}$. Then, if $\epsilon$ is an infinitesimal and\\

 $n\geq 2\eta(1+{log(\epsilon)\over log(4^{\eta}-1)-log(4^{\eta})})$ $(*)$\\

we have for any initial probability distribution $\pi_{0}$, that;\\

$\pi_{j}^{n}\simeq{1\over\eta}$ for $0\leq j\leq \eta-1$ $(**)$\\

If $\mu_{0}=\mu_{0}^{+}-\mu_{0}^{-}$ is a nonstandard distribution with sums $\{K^{+},K^{-}\}$, possibly infinite, then if $K=K^{+}-K^{-}$, and $\epsilon$ is an infinitesimal with $(K^{+}+K^{-})\epsilon\simeq 0$, and $n$ satisfies $(*)$, we obtain that;\\

$\mu_{j}^{n}\simeq{K\over\eta}$ for $0\leq j\leq \eta-1$ $(**)$\

\end{lemma}

\begin{proof}
Let $Seq_{1}=\{f:\mathcal{N}\rightarrow\mathcal{R}\}$ and $Seq_{2}=\{f:\mathcal{N}^{2}\rightarrow\mathcal{R}\}$. We let;\\

\noindent $Prob_{N}=\{f\in Seq_{1}:(\forall_{m\geq N}f(m)=0)\wedge(\forall_{0\leq m\leq N-1}f(m)\geq 0)\wedge\sum_{0\leq m\leq N-1}f(m)=1\}$.\\

encode probability vectors of length $N$. Let $G:\mathcal{N}\rightarrow Seq_{2}$ be defined by;\\

$G(N,0,2)=G(N,0,N-2)={1\over 2}$, $(\forall_{m\neq 2,N-2}G(N,0,m)=0$\\

$(\forall_{0\leq k\leq N-2}\forall_{1\leq m\leq N-1}(G(N,k+1,0)=G(N,k,N-1)$, $G(N,k+1,m)=G(N,k,m-1))$\\

$\forall_{k\geq N}\forall_{m\geq N}G(N,k,m)=0$\\

$G$ encodes the transition matrices for the given Markov chain with $N$ states. Let $H:\mathcal{N}^{2}\rightarrow Seq_{2}$ be defined by;\\

$(\forall_{0\leq i,j\leq N-1})H(1,N,i,j)=G(N,i,j)$\\

$(\forall_{i,j\geq N})H(1,N,i,j)=0$\\

$(\forall_{0\leq i,j\leq N-1}\forall_{n\geq 2})H(n,N,i,j)=\sum_{0\leq k\leq N-1}H(n-1,N,i,k)G(N,k,j)$\\

$(\forall_{i,j>N}\forall_{n\geq 2})H(n,N,i,j)=0$\\

$H$ encodes the powers $G(N)^{n}$ of the transition matrices. We define maps $L(N,n):Prob_{N}\rightarrow Prob_{N}$ by;\\

$(\forall_{0\leq j\leq N-1})L(N,n)(f)(j)=\sum_{0\leq k\leq N-1}f(k)H(n,N,k,j)$\\

$(\forall_{j\geq N})L(N,n)(f)(j)=0$\\

$L(N,n)(f)$ encodes the probability vectors $\pi^{n}$ for an initial distribution $\pi_{0}$ represented by $f$.\\

By a simple rearrangement, we have that the bound in $|\pi_{j}^{n}-{1\over N}|$, from Lemma \ref{transition}, can be formulated in first order logic as;\\

$\forall N\in\mathcal{N}_{odd}\forall\pi\in Prob_{N}\forall\epsilon\in\mathcal{R}_{>0}\forall{n\in\mathcal{N}}(n\geq 2N(1+{log(\epsilon)\over log(4^{N}-1)-log(4^{N})})\rightarrow(|L(n,N)(\pi)(j)-{1\over N}|\leq\epsilon, 0\leq j\leq N-1)$\\

By transfer, we obtain a corresponding result, quantifying over ${^{*}\mathcal{N}}$. Taking $\epsilon$ to be an infinitesimal and $\eta$ to be an infinite odd natural number, we obtain the first result. Observe that by construction of $G,H,L$, the nonstandard Markov chain with $\eta$ states evolves by the usual nonstandard matrix multiplication by the transition matrix, of the initial probability distribution. The remaining claim is similar and left to the reader.

\end{proof}

\begin{defn}
\label{measures}
Let $\eta\in {{^{*}\mathcal{N}}\setminus\mathcal{N}}$, infinite and odd, and let $\nu={\eta^{2}\over 2}$, $\nu\in{{^{*}\mathcal{Q}_{\geq 0}}\setminus\mathcal{Q}}$. We let;\\

$\overline{\Omega_{\eta}}=\{x\in{^{*}\mathcal{R}}:0\leq x<1\}$,  $\overline{\mathcal{T}_{\nu}}=\{t\in{^{*}\mathcal{R}_{\geq 0}}\}$\\

We let $\mathcal{C}_{\eta}$ consist of internal unions of the intervals $[{i\over\eta},{i+1\over\eta})$, for $0\leq i\leq \eta-1$, and let $\mathcal{D}_{\nu}$ consist of internal unions of $[{i\over\nu},{i+1\over\nu})$, for $ i\in{^{*}\mathcal{Z}_{\geq 0}}$.\\

We define counting measures $\mu_{\eta}$ and $\lambda_{\nu}$ on $\mathcal{C}_{\eta}$ and $\mathcal{D}_{\nu}$ respectively, by setting $\mu_{\eta}([{i\over\eta},{i+1\over\eta}))={1\over\eta}$, $\lambda_{\nu}([{i\over\nu},{i+1\over\nu}))={1\over\nu}$, for $0\leq i\leq \eta-1$, $i\in{^{*}\mathcal{Z}_{\geq 0}}$ respectively.\\

We let $(\overline{\Omega_{\eta}},\mathcal{C}_{\eta},\mu_{\eta})$ and $(\overline{\mathcal{T}}_{\nu},\mathcal{D}_{\nu},\lambda_{\nu})$ be the resulting measure spaces, in the sense of \cite{L}. We let $(\overline{\Omega_{\eta}}\times\overline{\mathcal{T}}_{\nu}, \mathcal{C}_{\eta}\times\mathcal{D}_{\nu}, \mu_{\eta}\times \lambda_{\nu})$ denote the corresponding product space.\\

If $f\in V(\overline{\Omega_{\eta}}\times\overline{\mathcal{T}}_{\nu})$ is measurable, we define;\\

${\partial f\over \partial t}({i\over \eta},{j\over \nu})=\nu (f({i\over \eta},{j+1\over \nu})-f({i\over \eta},{j\over \nu}))$, ${\partial f\over \partial t}(x,s)={\partial f\over \partial t}({[\eta x]\over \eta},{[\nu s]\over \nu})$\\

${\partial f\over \partial x}({i\over \eta},{j\over \nu})={\eta\over 2}(f({i+1\over \eta},{j\over \nu})-f({i-1\over \eta},{j\over \nu}))$, ${\partial f\over \partial x}(y,t)={\partial f\over \partial x}({[\eta y]\over \eta},{[\nu t]\over \nu})$\\

where we adopt the usual convention of taking $i$ mod $\eta$.

\end{defn}

\begin{defn}
\label{initial}
Let $f:\overline{\Omega_{\eta}}\rightarrow {^{*}\mathcal{R}}$  be measurable with respect to the $*\sigma$-algebra $\mathcal{C}_{\eta}$, in the sense of \cite{L}. We define $F:\overline{\Omega_{\eta}}\times \overline{\mathcal{T}_{\nu}}\rightarrow{^{*}\mathcal{R}}_{\geq 0}$ by;\\

$F({i\over\eta},{j\over\nu})=(\pi_{f}K^{j})(i)$, for $0\leq i\leq \eta-1$, $j\in{^{*}\mathcal{Z}_{\geq 0}}$\\

$F(x,t)=F({[\eta x]\over\eta},{[\nu t]\over\nu})$, $(x,t)\in \overline{\Omega_{\eta}}\times \overline{\mathcal{T}_{\nu}}$\\

where $\pi_{f}$ is the nonstandard distribution vector corresponding to $f$, $K$ is the transition matrix of the above Markov chain with $\eta$ states, and $K^{j}$ denotes a nonstandard power.

\end{defn}

\begin{theorem}
\label{measurable}
Let $F$ be as defined in Definition \ref{initial}, then $F$ is measurable with respect to $\mathcal{C}_{\eta}\times\mathcal{D}_{\nu}$, and, moreover $F$ is the unique solution to the nonstandard heat equation;\\

${\partial F\over \partial t}-{\partial^{2}f\over \partial x^{2}}=0$\\

with initial condition $f$. If $f$ is bounded, then for $\tau\geq {16(4^{\eta})log(\eta)\over \eta}$, we have that $F_{\tau}\simeq C$, where $C=\int_{\overline{\Omega}_{\eta}}fd\mu_{\eta}$.\\
\end{theorem}

\begin{proof}
The first proposition follows by observing that the defining schema for $F$ is internal and by hyperfinite induction, see Lemma \ref{nonstandard1} for the mechanics of this transfer process. For the second proposition, it is a simple computation, using the definition of the partial derivatives in Definition \ref{measures}, to see that, if $F$ satisfies the nonstandard heat equation, then;\\

$F({i\over \eta},{j+1\over \nu})={\eta^{2}\over 4\nu}F({i+2\over\eta},{j\over \nu})+(1-{\eta^{2}\over 2\nu})F({i\over\eta},{j\over \nu})+{\eta^{2}\over 4\nu}F({i-2\over\eta},{j\over\nu})$, $j\in{^{*}\mathcal{Z}_{\geq 0}}$\\

In particular, $F$ is uniquely determined from the initial condition $f$ and taking $\eta^{2}=2\nu$, we obtain that;\\

$F({i\over \eta},{j+1\over \nu})={1\over 2}F({i+2\over\eta},{j\over \nu})+{1\over 2}F({i-2\over\eta},{j\over\nu})$, $j\in{^{*}\mathcal{Z}_{\geq 0}}$\\

which agrees with the defining schema for $F$ in Definition \ref{initial}. For the last claim, by definition of the nonstandard integral, see \cite{dep}, and the assumptions on $f$, we have that $f=f^{+}-f^{-}$, with corresponding sums $\{K^{+},K^{-},K\}$, where ${(K^{+}+K^{-})\over \eta^{2}}\simeq 0$, and $\int_{\overline{\Omega}_{\eta}}fd\mu_{\eta}={K\over \eta}$. By Lemma \ref{nonstandard1}, we have, taking $\epsilon={1\over \eta^{2}}$, that for;\\

$n\geq 2\eta(1-{log(\eta^{2})\over log(4^{\eta-1})-log(4^{\eta})})$\\

$F_{n\over \nu}\simeq \int_{\overline{\Omega}_{\eta}}fd\mu_{\eta}$. Then, we compute;\\

$2\eta(1-{log(\eta^{2})\over log(4^{\eta-1})-log(4^{\eta})})$\\

$\leq 4\eta({log(\eta)\over log(4^{\eta})-log(4^{\eta}-1)})+1$\\

$=4\eta({log(\eta)\over log(1+{1\over 4^{\eta}-1})})+1$\\

$\leq 8\eta(4^{\eta}-1)log(\eta)$ as $log(1+x)\geq {x\over 2}$, for $x\simeq 0$\\

It follows that, for ${n\over \nu}\geq {16\over\eta}(4^{\eta}-1)log(\eta)$, $F_{n\over \nu}\simeq \int_{\overline{\Omega}_{\eta}}fd\mu_{\eta}$, therefore, if $\tau\geq {16\over\eta}(4^{\eta})log(\eta)$, $F_{\tau}\simeq \int_{\overline{\Omega}_{\eta}}fd\mu_{\eta}$, as required.

\end{proof}

We now give an alternative description of the process given in Theorem \ref{measurable}. Namely, we can think of it as the density of a collection of particles , moving independently and at random. For sufficiently large $t$, the density, which we refer to as the equilibrium density, is close to being constant. This idea is made precise in the following.

\begin{defn}
\label{independent}
We keep the notation of Definition \ref{measures}. We let $\nu={\eta^{2}\over 2}$ but we drop the restriction that $\eta$ is odd. We let;\\

$\overline{\Omega}_{\kappa}=\{(s_{i}):1\leq i\leq \kappa, s_{i}=1\ or\ -1\}$\\

so that ${^{*}Card}(\overline{\Omega}_{\kappa})=2^{\kappa}$. We let;\\

$\omega_{i}:\overline{\Omega}_{\kappa}\rightarrow \{1,-1\}$, for $1\leq i\leq \kappa$, be defined by;\\

$\omega_{i}(s)=s_{i}$\\

We let;\\

$\overline{\mathcal{T}_{\nu,\kappa}}=\{t\in\overline{\mathcal{T}_{\nu}}:0\leq [\nu t]\leq \kappa\}$\\

We let $\chi:\overline{\Omega}_{\kappa}\times \overline{\mathcal{T}_{\nu,\kappa}}\rightarrow\overline{\Omega}_{\eta}$, be defined by;\\

$\chi(s,t)={1\over \eta}({^{*}}\sum_{j=1}^{[\nu t]}\omega_{j}(x))$ $mod [0,1)$, $1\leq [\nu t]\leq \kappa$\\

$\chi(s,0)=0$\\

We let $\overline{\chi}:\overline{\Omega_{\eta}}\times \overline{\Omega_{\kappa}}\times\overline{\mathcal{T}_{\nu,\kappa}}\rightarrow\overline{\Omega_{\eta}}$ be defined by;\\

$\overline{\chi}(x,s,t)=x+2\chi(s,t)$ $mod [0,1)$\\

Given an initial condition $f\in V(\overline{\Omega_{\eta}})$, with $f\geq 0$, we let;\\

$N_{f}:\overline{\Omega_{\eta}}\times\overline{\mathcal{T}_{\nu,\kappa}}\rightarrow {^{*}\mathcal{R}_{\geq 0}}$ be defined by;\\

$N_{f}(x,t)={^{*}\sum}_{0\leq i\leq \eta-1}{f({i\over\eta})\over 2^{\kappa}}{^{*}Card}(\{s\in\overline{\Omega}_{\kappa}:\overline{\chi}({i\over\eta},s,t)={[\eta x]\over \eta}\})$\\

\end{defn}

\begin{lemma}
\label{equivalent}
Let $f\in V(\overline{\Omega}_{\eta})$, $f\geq 0$ be an initial condition, for the heat equation in Theorem \ref{measurable} or the Markov chain in Definition \ref{initial}, then $N_{f}$ as given in Definition \ref{independent} is exactly the process $F$ given by Lemma \ref{measurable}.
\end{lemma}

\begin{proof}
This follows easily by hyperfinite induction. As both $N_{f}$ and $F$ are measurable on $\overline{\Omega}_{\eta}\times\overline{\mathcal{T}_{\nu,\kappa}}$, it is sufficient to check the two claims that;\\

 $N_{f}(x,0)=f(x)$\\

 $N_{f}(x,{j+1\over\nu})={1\over 2}N_{f}(x+{2\over\eta},{j\over\nu})+{1\over 2}N_{f}(x-{2\over\eta},{j\over\nu})$\\

for $0\leq j\leq \kappa-1$, $x\in\overline{\Omega_{\eta}}$. For the first claim, observe that, if $[\eta x]=i$, then $\overline{\chi}({i\over\eta},s,0)={[\eta x]\over\eta}$ for all $s\in\overline{\Omega}_{\kappa}$, and if $[\eta x]\neq i$, then $\overline{\chi}({i\over\eta},s,0)={[\eta x]\over\eta}$ for no $s\in\overline{\Omega}_{\kappa}$, by definition of $\overline{\chi}$. Hence, a simple computation of $N_{f}(x,0)$ gives the result. For the second claim, just observe that, for $0\leq i\leq \eta-1$, $0\leq j\leq \kappa-1$ ;\\

${^{*}Card(s\in\overline{\Omega_{\kappa}}:\overline{\chi}({i\over\eta},s,{j+1\over\nu})={[\eta x]\over \eta})}$\\

$={1\over 2}{^{*}Card(s\in\overline{\Omega_{\kappa}}:\overline{\chi}({i\over\eta},s,{j\over\nu})={[\eta x]+2\over \eta})}+{1\over 2}{^{*}Card(s\in\overline{\Omega_{\kappa}}:\overline{\chi}({i\over\eta},s,{j\over\nu})={[\eta x]-2\over \eta})}$\\

The second claim then follows by linearity and the definition of $N_{f}$.

\end{proof}
\begin{defn}
\label{reverse}
Let $(\overline{\Omega_{\eta}},\mathcal{E}_{\eta},\gamma_{\eta})$ be a nonstandard $*$-finite measure space. We define a reverse filtration on $\overline{\Omega_{\eta}}$ to be an internal collection of $*\sigma$-algebras $\mathcal{E}_{\eta,i}$, indexed by $0\leq i\leq\kappa$, $\kappa\in {{^{*}\mathcal{N}}\setminus \mathcal{N}}$, such that;\\

(i). $\mathcal{E}_{\eta,0}=\mathcal{E}_{\eta}$\\

(ii). $\mathcal{E}_{\eta,i}\subseteq \mathcal{E}_{\eta,j}$, if $0\leq j\leq i\leq \kappa$.\\

We say that $\overline{F}:\overline{\Omega_{\eta}}\times\overline{\mathcal{T}_{\nu,\kappa}}\rightarrow {^{*}\mathcal{R}}$ is adapted to the filtration if $\overline{F}$ is measurable with respect to $\mathcal{E}_{\eta}\times \mathcal{D}_{\nu}$ and $\overline{F}_{i\over\nu}:\overline{\Omega_{\eta}}\rightarrow{^{*}\mathcal{R}}$ is measurable with respect to $\mathcal{E}_{\eta,i}$, for $0\leq i\leq \kappa$.\\

If $f:\overline{\Omega_{\eta}}\rightarrow{^{*}\mathcal{R}}$ is measurable with respect to $\mathcal{E}_{\eta,j}$ and $0\leq j\leq i\leq \kappa$, we define the conditional expectation $E_{\eta}(f|\mathcal{E}_{\eta,i})$ to be the unique $g:\overline{\Omega_{\eta}}\rightarrow{^{*}\mathcal{R}}$ such that $g$ is measurable with respect to $\mathcal{E}_{\eta,i}$ and;\\

$\int_{U}gd\gamma_{\eta}=\int_{U}f d\gamma_{\eta}$\\

for all $U\in \mathcal{E}_{\eta,i}$. We say that  $\overline{F}:\overline{\Omega_{\eta}}\times\overline{\mathcal{T}_{\nu,\kappa}}\rightarrow {^{*}\mathcal{R}}$ is a reverse martingale if;\\

(i). $\overline{F}$ is adapted to the reverse filtration on $\overline{\Omega_{\eta}}$\\

(ii). $E_{\eta}(\overline{F}_{j\over\nu}|\mathcal{E}_{\eta,i})=\overline{F}_{i\over\nu}$ for $0\leq j\leq i\leq \kappa$\\

\end{defn}

\begin{theorem}
\label{martingale}
Let $F$ be as in Definition \ref{initial}, without the restriction that $\eta$ is odd, but keeping $\nu={\eta^{2}\over 2}$, and let $F_{\kappa}$ be its restriction to $\overline{\Omega}_{\eta}\times \overline{\mathcal{T}_{\nu,\kappa}}$. Then there exists a reverse filtration on $\overline{\Omega_{\eta}}$ and $\overline{F}_{\kappa}$ such that $\overline{F}_{\kappa}$ is a reverse martingale, and $\overline{F}_{\kappa,{\kappa\over\nu}}=F_{{\kappa\over \nu}}$

\end{theorem}

\begin{proof}
We define the reverse filtration, by setting $\mathcal{E}_{\eta,i}$ to be internal unions of the intervals $[{j\over 2^{\kappa-i}\eta},{j+1\over 2^{\kappa-i}\eta})$ for $0\leq j\leq 2^{\kappa-i}\eta-1$, $0\leq i\leq \kappa$. Clearly, this is an internal collection. It follows that $\mathcal{E}_{\eta}=\mathcal{E}_{\eta,0}$ consists of internal unions of the intervals $[{j\over 2^{\kappa}\eta},{j+1\over 2^{\kappa}\eta})$ for $0\leq j\leq 2^{\kappa}\eta-1$, and we define the corresponding measure $\gamma_{\eta}$ by setting $\gamma_{\eta}([{j\over 2^{\kappa}\eta},{j+1\over 2^{\kappa}\eta}))={1\over 2^{\kappa}\eta}$. Observe that $\mathcal{E}_{\eta,\kappa}=\mathcal{C}_{\eta}$, the original $*\sigma$-algebra.\\

We define bijections $\Phi_{i}:{^{*}\mathcal{N}_{0\leq j\leq \eta-1}}\times \overline{\Omega}_{\kappa-i}\rightarrow {^{*}\mathcal{N}_{0\leq j\leq 2^{\kappa-i}\eta-1}}$, for $0\leq i\leq \kappa$, where ${\overline{\Omega}_{\kappa-i}}=\{(\omega_{k}):\omega_{k}=1\ or -1, 1\leq k\leq \kappa-i\}$, by;\\

$\Phi_{i}(j,\omega)=2^{\kappa-i}j+ 2^{\kappa-i}{^{*}\sum_{1\leq k\leq \kappa-i}}{\omega_{k}+1\over 2^{k+1}}$\\

Define $\overline{F}_{\kappa}$ by;\\

$\overline{F}_{\kappa}({r\over 2^{\kappa-i}\eta},{i\over\nu})=F_{i\over\nu}({j\over\eta}+{2\over\eta}{^{*}\sum}_{1\leq k\leq \kappa-i}\omega_{k})$\\

where $\Phi_{i}(j,\omega)=r$, for $0\leq r\leq 2^{\kappa-i}\eta-1$, $0\leq i\leq \kappa$.\\

$\overline{F}_{\kappa}(x,t)=\overline{F}_{\kappa}({[2^{\kappa-[\nu t]}\eta x]\over 2^{\kappa-[\nu t]}\eta},{[\nu t]\over \nu})$, $(x,t)\in\overline{\Omega}_{\eta}\times \overline{\mathcal{T}_{\nu,\kappa}}$\\

It is clear that $\overline{F}_{\kappa}$ is adapted to the reverse filtration on $\overline{\Omega}_{\eta}$. Moreover, it is straightforward to see that;\\

$\overline{F}_{\kappa}({r\over\eta},{\kappa\over \nu})=F_{\kappa\over\nu}({r\over\eta})$\\

as $\Phi_{\kappa}(r)=r$, so $\overline{F}_{\kappa,{\kappa\over \nu}}=F_{\kappa\over \nu}$. We claim that $\overline{F}_{\kappa}$ is a reverse martingale. We have verified condition $(i)$ in Definition \ref{reverse}. To verify $(ii)$, by the tower law for conditional expectation, it is sufficient to prove that $E_{\eta}(\overline{F}_{\kappa,{i\over\nu}}|\mathcal{E}_{i+1})=\overline{F}_{\kappa,{i+1\over\nu}}$, for $0\leq i\leq \kappa-1$. We have that;\\

$E_{\eta}(\overline{F}_{\kappa,{i\over\nu}}|\mathcal{E}_{i+1})({r\over 2^{\kappa-i-1}\eta})$\\

$=2^{\kappa-i-1}\eta\int_{[{r\over 2^{\kappa-i-1}\eta},{r+1\over 2^{\kappa-i-1}\eta})}E_{\eta}(\overline{F}_{\kappa,{i\over\nu}}|\mathcal{E}_{i+1})d\gamma_{\eta}$\\

$=2^{\kappa-i-1}\eta\int_{[{r\over 2^{\kappa-i-1}\eta},{r+1\over 2^{\kappa-i-1}\eta})}\overline{F}_{\kappa,{i\over\nu}}d\gamma_{\eta}$\\

$={2^{\kappa-i-1}\eta\over 2^{\kappa-i}\eta}(\sum_{m=0}^{1}\overline{F}_{\kappa,{i\over\nu}}({2r+m\over 2^{\kappa-i}\eta}))$\\

$={1\over 2}(F_{i\over\nu}({j\over\eta}+{2\over\eta}({^{*}\sum_{1\leq k\leq \kappa-i-1}}\omega_{k}-1))+F_{i\over\eta}({j\over\eta}+{2\over\eta}({^{*}\sum_{1\leq k\leq \kappa-i-1}}\omega_{k}+1)))$\\

$={1\over 2}(F_{i\over\nu}(x-{2\over \eta})+F_{i\over\nu}(x+{2\over\eta}))$\\

$=F_{i+1\over\nu}(x)=\overline{F}_{\kappa,{i+1\over\nu}}({r\over 2^{\kappa-i-1}\eta})$\\

where $\Phi_{i+1}(j,\omega)=r$, $\omega=(\omega_{k})_{1\leq k\leq \kappa-i-1}$ and $x={j\over\eta}+{2\over\eta}({^{*}\sum}_{1\leq k\leq \kappa-i-1}\omega_{k})$, as required.\\

\end{proof}

We now require a lemma about the rate of convergence in the Central Limit Theorem for a particular class of random variable.\\

\begin{lemma}
\label{centrallimit}
Let $\kappa$ be odd and infinite, and let $\{X_{i}:1\leq i\leq \kappa\}$ be discrete identically distributed random variables $X_{i}:\Omega_{\kappa}\rightarrow{^{*}\mathcal{R}}$ which are $*$-independent with respect to the probability measure $P=\mu_{\kappa}$, and take the value $\sqrt{2t}$ with probability ${1\over 2}$ and $-\sqrt{2t}$ with probability ${1\over 2}$, where $t>0$. Then if $T_{\kappa}={{^{*}\sum}_{1\leq i\leq \kappa}X_{i}\over \sqrt{\kappa}}$, there exists a finite constant $L$, such that;\\

$|P(T_{\kappa}={\sqrt{2t}j\over\sqrt{\kappa}})-{\sqrt{2}\over \sqrt{\pi\kappa}}{^{*}exp}({-j^{2}\over 2\kappa})|\leq L\kappa^{{-3\over 2}}$\\

for $j$ odd and $-\kappa\leq j\leq \kappa$.
\end{lemma}

\begin{proof}
This result will be obtained by transfer from the finite case. We take $n$ odd and finite, and consider the iid random variables $\{X_{1},\ldots,X_{n}\}$ on $\Omega_{n}$, where $X_{i}$ takes the value $1$ with probability ${1\over 2}$ and $-1$ with probability ${1\over 2}$. We have that $E(X_{i})=0$, $E(X_{i}^{2})=1$. Let $S_{n}={X_{1}+\ldots+X_{n}\over\sqrt{n}}$. We also have that;\\

$\phi_{S_{n}}(x)=E(e^{ixS_{n}})$\\

$=\sum_{-n\leq j odd\leq n}P(S_{n}={j\over\sqrt{n}})e^{{ixj\over\sqrt{n}}}$\\

$=2\sum_{0\leq j odd\leq n}P(S_{n}={j\over\sqrt{n}})cos({xj\over\sqrt{n}})$ $(*)$\\

As is easily checked, we have the identity;\\

${2\over\pi}\int_{-{\pi\over 2}}^{{\pi\over 2}}cos(tx)cos(sx)dx=\delta_{ts}$, for $t,s$ odd integers.\\

Making the substitution $x={y\over\sqrt{n}}$, we obtain;\\

${2\over \pi\sqrt{n}}\int_{{-\pi\sqrt{n}\over 2}}^{{\pi\sqrt{n}\over 2}}cos({ty\over\sqrt{n}})cos({sy\over\sqrt{n}})dy=\delta_{ts}$, for $t,s$ odd integers.\\

Combining this with $(*)$, we obtain that;\\

$P(S_{n}={j\over\sqrt{n}})={1\over \pi\sqrt{n}}\int_{{-\pi\sqrt{n}\over 2}}^{{\pi\sqrt{n}\over 2}}cos({jx\over\sqrt{n}})\phi_{S_{n}}(x)dx$ for $j$ an odd integer.\\

Now, by independence, and a simple calculation of the moments of $X$, we have that;\\

$\phi_{S_{n}}(x)=\phi_{X}({x\over\sqrt{n}})^{n}$\\

$=E(e^{{ixX\over \sqrt{n}}})^{n}$\\

$=e^{nlog(1-h_{n}(x))}$\\

where;\\

$h_{n}(x)={x^{2}\over 2!n}-{x^{4}\over 4!n^{2}}+\ldots$\\

It is easy to check that $h_{n}(x)=1-cos({x\over\sqrt{n}})$. In particular, in the range $|x|< {\pi\sqrt{n}\over 2}$, we have that $0\leq h_{n}(x)<1$. It follows that, for $|x|<{\pi\sqrt{n}\over 2}$, using the power series expansion of $log(1-h_{n}(x))$, that;\\

$\phi_{S_{n}}(x)-e^{{-x^{2}\over 2}}=e^{{-x^{2}\over 2}}(e^{\alpha_{n}(x)-\beta_{n}(x)}-1)$\\

where;\\

$\alpha_{n}(x)={x^{4}\over 4!n}-{x^{6}\over 6!n^{2}}+\ldots$\\

$\beta_{n}(x)={nh_{n}(x)^{2}\over 2}+{nh_{n}(x)^{3}\over 3}+\ldots$\\

We claim that $\alpha_{n}(x)\leq \beta_{n}(x)$, for $|x|<{\pi\sqrt{n}\over 2}$, $(**)$. In order to see this, let $h(x)=log(e^{{x^{2}\over 2}}cos(x))$, then $h'(x)=x-tan(x)$, $h''(x)=-tan^{2}(x)$, so $h'(x)=x-tan(x)\leq 0$, for $x\in [0,{\pi\over 2})$, and, therefore, $h(x)\leq h(0)=0$, for $x\in [0,{\pi\over 2})$. It follows that $cos(x)\leq e^{{-x^{2}\over 2}}$, for $|x|<{\pi\over 2}$. By substitution, we have that $cos({x\over \sqrt{n}})\leq e^{{-x^{2}\over 2n}}$, for $|x|<{\pi\sqrt{n}\over 2}$. Then, $1-h_{n}(x)\leq e^{{-x^{2}\over 2n}}$, so $nlog(1-h_{n}(x))+{x^{2}\over 2}\leq 0$. Using the fact that $nlog(1-h_{n}(x))=\alpha_{n}(x)-\beta_{n}(x)-{x^{2}\over 2}$, we obtain the result $(**)$ as required.\\

Now using the fact that $\beta_{n}(x)\leq {nh_{n}^{2}\over 1-h_{n}}$, the identity $1+h\leq e^{h}$, for $h\leq 0$, and $h_{n}(x)\leq {x^{2}\over 2n}$, $\alpha_{n}(x)\leq {x^{4}\over 4!n}$, for $|x|<{\pi\sqrt{n}\over 2}$, we obtain, for $|x|<{\pi\sqrt{n}\over 2}$, that;\\

$|e^{-x^{2}\over 2}(e^{\alpha_{n}(x)-\beta_{n}(x)}-1)|$\\

$\leq |e^{-x^{2}\over 2}(e^{\alpha_{n}(x)-{nh_{n}^{2}(x)\over 1-h_{n}(x)}}-1)|$\\

$\leq |e^{-x^{2}\over 2}(\alpha_{n}(x)-{nh_{n}^{2}(x)\over 1-h_{n}(x)})|$\\

$\leq e^{-x^{2}\over 2}({x^{4}\over 4!n}+{x^{4}\over 4ncos({x\over\sqrt{n}})})$\\

We have that;\\

$\int_{-\pi\sqrt{n}\over 3}^{\pi\sqrt{n}\over 3}e^{-x^{2}\over 2}{x^{4}\over 4ncos({x\over\sqrt{n}})}dx$\\

$\leq \int_{-\pi\sqrt{n}\over 3}^{\pi\sqrt{n}\over 3}{e^{-x^{2}\over 2}x^{4}\over 2n}dx$\\

$={1\over 2n}(-({2\pi^{3}n^{3\over 2}\over 9}+2\pi\sqrt{n})e^{{-\pi^{2}n\over 18}}+3\int_{-\pi\sqrt{n}\over 3}^{\pi\sqrt{n}\over 3}e^{-x^{2}\over 2}dx)$ (integrating by parts)\\

$\leq {1\over 2n}(-({2\pi^{3}n^{3\over 2}\over 9}+2\pi\sqrt{n})e^{{-\pi^{2}n\over 18}}+3\sqrt{2\pi})$\\

$\leq {C\over n}$ $(\dag)$\\

 for $n$ sufficiently large and $C\geq 0$ a finite constant. Now for $|j|\leq n$, $j$ odd, using L'Hopital's rule, we have that;\\

$|lim_{x\rightarrow {\pi\sqrt{n}\over 2}}{|cos({jx\over\sqrt{n}})|\over cos({x\over\sqrt{n}})}|=|j|\leq n$\\

It follows that;\\

$\int_{{\pi\sqrt{n}\over 3}\leq |x|<{\pi\sqrt{n}\over 2}}e^{-x^{2}\over 2}{|cos({jx\over\sqrt{n}})|x^{4}\over 4ncos({x\over\sqrt{n}})}dx$\\

$\leq \int_{{\pi\sqrt{n}\over 3}\leq |x|<{\pi\sqrt{n}\over 2}}{ne^{-x^{2}\over 2}x^{4}\over 4n}dx$\\

$\leq {2\pi\sqrt{n}\over 6}{\pi^{4}n^{2}\over 324}e^{-\pi^{2}n\over 9}$\\

$={\pi^{5}n^{5\over 2}e^{{-\pi^{2}n\over 9}}\over 972}\leq {D\over n}$ $(\dag\dag)$\\

for $n$ sufficiently large, and $D\geq 0$ a finite constant. Combining $(\dag)$ and $(\dag\dag)$, we obtain that;\\

$\int_{{-\pi\sqrt{n}\over 2}}^{{\pi\sqrt{n}\over 2}}e^{-x^{2}\over 2}{|cos({jx\over\sqrt{n}})|x^{4}\over 4ncos({x\over\sqrt{n}})}dx\leq {E\over n}$\\

where $E=C+D$. In the same way, we can find a finite constant $F\geq 0$, for which;\\

$\int_{{-\pi\sqrt{n}\over 2}}^{{\pi\sqrt{n}\over 2}}e^{-x^{2}\over 2}|cos({jx\over\sqrt{n}})|({x^{4}\over 4!n}+{x^{4}\over 4ncos({x\over\sqrt{n}})})dx\leq {F\over n}$\\

Combining these inequalities, we obtain;\\

${1\over \pi\sqrt{n}}\int_{{-\pi\sqrt{n}\over 2}}^{{\pi\sqrt{n}\over 2}}|\phi_{S_{n}(x)}cos({jx\over\sqrt{n}})-e^{-x^{2}\over 2}cos({jx\over\sqrt{n}})|dx$\\

$\leq {F\over \pi}n^{-3\over 2}$\\

$=Gn^{-3\over 2}$\\

for $n$ sufficiently large and $G={F\over\pi}$ a finite constant. We now have that;\\

$|P(S_{n}={j\over\sqrt{n}})-{1\over \pi\sqrt{n}}\int_{{-\pi\sqrt{n}\over 2}}^{{\pi\sqrt{n}\over 2}}e^{-x^{2}\over 2}cos({jx\over\sqrt{n}})dx|\leq Gn^{-3\over 2}$\\

Therefore;\\

$|P(S_{n}={j\over\sqrt{n}})-{1\over \pi\sqrt{n}}\int_{-\infty}^{\infty}e^{-x^{2}\over 2}cos({jx\over\sqrt{n}})dx|$\\

$\leq Gn^{-3\over 2}+|{1\over \pi\sqrt{n}}\int_{|x|>{\pi\sqrt{n}\over 2}}e^{-x^{2}\over 2}cos({jx\over\sqrt{n}})dx|$\\

$\leq Gn^{-3\over 2}+{1\over \pi\sqrt{n}}\int_{|x|>{\pi\sqrt{n}\over 2}}e^{-x^{2}\over 2}dx$\\

We have that $e^{-x^{2}\over 2}\leq H|x|^{-3}$, for $|x|>{\pi\sqrt{n}\over 2}$, $n$ sufficiently large and $H\geq 0$ a finite constant. Using this inequality, and performing the integration, gives;

$|P(S_{n}={j\over\sqrt{n}})-{1\over \pi\sqrt{n}}\int_{-\infty}^{\infty}e^{-x^{2}\over 2}cos({jx\over\sqrt{n}})dx|$\\

$\leq Gn^{-3\over 2}+{2H\over \pi\sqrt{n}}{2^{2}\over 2\pi^{2}n}$\\

$=Gn^{-3\over 2}+Kn^{-3\over 4}=Ln^{-3\over 2}$\\

where $K={4H\over \pi^{3}}$ and $L=G+K$. Now we have that;\\

${1\over \pi\sqrt{n}}\int_{-\infty}^{\infty}e^{-x^{2}\over 2}cos({jx\over\sqrt{n}})dx$\\

$={1\over \pi\sqrt{n}}\mathcal{F}(e^{-x^{2}\over 2})|_{{j\over\sqrt{n}}}$\\

$={1\over \pi\sqrt{n}}\sqrt{2\pi}e^{-y^{2}\over 2}|_{{j\over\sqrt{n}}}$\\

$={\sqrt{2}\over \sqrt{\pi n}}e^{-j^{2}\over 2n}$\\

where $\mathcal{F}$ denotes the Fourier transform. So we obtain that;\\

$|P(S_{n}={j\over\sqrt{n}})-{\sqrt{2}\over \sqrt{\pi n}}e^{-j^{2}\over 2n}|\leq Ln^{-3\over 2}$ $(\dag\dag\dag)$\\

Consider now the case, when the original $\{X_{i}:1\leq i\leq n\}$ on $\Omega_{n}$ take the values $\sqrt{2t}$ with probability ${1\over 2}$ and $-\sqrt{2t}$ with probability ${1\over 2}$, where $t\in\mathcal{R}_{>0}$. Then $\{{X_{i}\over\sqrt{2t}}:1\leq i\leq n\}$ are as in the above proof, so we can apply the result $(\dag\dag\dag\dag)$, to obtain that;\\

$|P(T_{n}={\sqrt{2t}j\over\sqrt{n}})-{\sqrt{2}\over \sqrt{\pi n}}e^{-j^{2}\over 2n}|\leq Ln^{-3\over 2}$\\

where $T_{n}={X_{1}+\ldots X_{n}\over \sqrt{n}}$. Now we can transfer this result from $\Omega_{n}$, with the measure $\mu_{n}$, for finite $n$, to the case when $\kappa$ is infinite, and obtain the result of the lemma.

\end{proof}

We now obtain an explicit solution for the nonstandard heat equation;\\

\begin{theorem}
\label{solution}
Let assumptions be as in Theorem \ref{martingale}, with $F_{0}$ S-continuous, then for $\kappa$ odd and ${^{\circ}{\kappa\over\nu}}\neq 0$, we have that;\\

$F_{{\kappa\over\nu}}({j\over\eta})={1\over\eta'}{^{*}\sum_{i'\in I}}F_{0,p}({j\over\eta}+{i'\over\eta'}){1\over 2\sqrt{\pi t}}{^{*}exp}({-({i'\over\eta'})^{2}\over 4t})$\\

for $0\leq j\leq \eta-1$, where $\eta'={\eta\over 4}$, $I'={^{*}\mathcal{Z}}\cap [-m,m]$, for some infinite integer $m\in{^{*}\mathcal{Z}}_{\geq 0}$, and $F_{0,p}$ is the periodic extension of $F_{0}$ to ${^{*}\mathcal{R}}$.

\end{theorem}

\begin{proof}
We have that;\\

$F_{{\kappa\over\nu}}({j\over\eta})=\eta\int_{[{j\over\eta},{j+1\over\eta})}F_{{\kappa\over\nu}}d\gamma_{\eta}$\\

$=\eta\int_{[{j\over\eta},{j+1\over\eta})}\overline{F}_{\kappa,{\kappa\over\nu}}d\gamma_{\eta}$\\

$=\eta\int_{[{j\over\eta},{j+1\over\eta})}\overline{F}_{\kappa,0}d\gamma_{\eta}$\\

=$\eta{1\over 2^{\kappa}\eta}{^{*}\sum}_{0\leq s\leq 2^{\kappa}-1}\overline{F}_{\kappa,0}({j\over\eta}+{s\over 2^{\kappa}\eta})$\\

$={1\over 2^{\kappa}}{^{*}\sum}_{\omega\in\overline{\Omega}_{\kappa}}F_{0}({j\over\eta}+{2\over \eta}{^{*}\sum}_{1\leq k\leq \kappa}\omega_{k})$\\

$={1\over 2^{\kappa}}{^{*}\sum}_{\omega\in\overline{\Omega}_{\kappa}}F_{0}({j\over\eta}+{\sqrt{2t}\over\sqrt{\kappa}}{^{*}\sum}_{1\leq k\leq \kappa}\omega_{k})$, $\eta^{2}=2\nu$, ${\kappa\over\nu}=t$\\

Letting $\mu_{\kappa}$ be the measure on $\overline{\Omega}_{\kappa}$, defined by $\mu_{\kappa}(\omega)={1\over 2^{\kappa}}$, we have that the random variables $\omega_{k,t}:\overline{{\Omega}_{\kappa}}\rightarrow{^{*}\mathcal{R}}$ defined by $\omega_{k,t}(\omega)=\sqrt{2t}\omega_{k}$ have the property of $*$-independence and satisfy the hypotheses of Lemma \ref{centrallimit}. We let;\\

$\overline{\mathcal{R}_{\eta, \sqrt{2t}\kappa}}=\{x\in{^{*}\mathcal{R}}:-{[(\sqrt{2t\kappa}+1)\eta]\over\eta}\leq x<{[(\sqrt{2t\kappa}+1)\eta]\over\eta}\}$\\

and let;\\

$\mu_{\eta,\sqrt{2t}\kappa}([{i\over\eta},{i+1\over\eta}))={1\over\eta}$, for $i\in I$\\

where $I={^{*}\mathcal{Z}}\cap [-[(\sqrt{2t\kappa}+1)\eta], [(\sqrt{2t\kappa}+1)\eta]-1]$. We let $F_{0,p}$ be the periodic extension of $F_{0}$ to $\overline{\mathcal{R}_{\eta, \sqrt{2t}\kappa}}$. Then, we compute;\\

$F_{\kappa\over\nu}({j\over\eta})={1\over 2^{\kappa}}{^{*}\sum}_{\omega\in\overline{\Omega}_{\kappa}}F_{0}({j\over\eta}+{\sqrt{2t}\over\sqrt{\kappa}}{^{*}\sum}_{1\leq k\leq \kappa}\omega_{k})$\\

$={1\over 2^{\kappa}}{^{*}\sum_{i\in I}}F_{0,p}({j\over\eta}+{i\over\eta}){^{*}\sum}_{\omega\in\overline{\Omega}_{\kappa}}({\sqrt{2t}\over\sqrt{\kappa}}{^{*}\sum}_{1\leq k\leq \kappa}\omega_{k})\in [{i\over\eta},{i+1\over\eta})$\\

$={^{*}\sum_{i\in I}}F_{0,p}({j\over\eta}+{i\over\eta})\mu_{\kappa}(\{\omega:{\sqrt{2t}\over\sqrt{\kappa}}{^{*}\sum}_{1\leq k\leq \kappa}\omega_{k}\in [{i\over\eta},{i+1\over\eta})\})$\\

Observing that ${1\over\eta}={\sqrt{t}\over \sqrt{2\kappa}}$, and using Lemma \ref{centrallimit}, we have that;\\

$\mu_{\kappa}(\{\omega:{\sqrt{2t}\over\sqrt{\kappa}}{^{*}\sum}_{1\leq k\leq \kappa}\omega_{k}\in [{i\over\eta},{i+1\over\eta})\})$\\

$=\mu_{\kappa}(\{\omega:T_{\kappa}(\omega)\in [{i\sqrt{t}\over\sqrt{2\kappa}},{(i+1)\sqrt{t}\over\sqrt{2\kappa}})\})$\\

$=\mu_{\kappa}(\{\omega:T_{\kappa}(\omega)={\sqrt{2t}j\over\sqrt{\kappa}},{i\sqrt{t}\over\sqrt{2\kappa}}\leq {\sqrt{2t}j\over\sqrt{\kappa}}<{(i+1)\sqrt{t}\over\sqrt{2\kappa}}\})$\\

$=\mu_{\kappa}(\{\omega:T_{\kappa}(\omega)={\sqrt{2t}j\over\sqrt{\kappa}},j={i\over 2}\})$\\

$={\sqrt{2}\over\sqrt{\pi\kappa}}{^{*}exp}({-i^{2}\over 8\kappa})+\epsilon$\\

where $|\epsilon|\leq L\kappa^{-3\over 2}$. Observing that ${^{*}Card}(I)\leq 6\kappa$, so $\epsilon{^{*}Card}(I)\simeq 0$, letting $I_{res}=\{i\in I:{i\over 2}\ is\ odd\}$, and replacing $\kappa$ by ${t\eta^{2}\over 2}$,  it follows that;\\

$F_{\kappa\over\nu}({j\over\eta})\simeq {^{*}\sum_{i\in I_{res}}}F_{0,p}({j\over\eta}+{i\over\eta}){\sqrt{2}\over\sqrt{\pi\kappa}}{^{*}exp}({-i^{2}\over 8\kappa})$\\

$={1\over\eta}{^{*}\sum_{i\in I_{res}}}F_{0,p}({j\over\eta}+{i\over\eta}){\sqrt{2}\eta\over\sqrt{\pi\kappa}}{^{*}exp}({-i^{2}\over 8\kappa})$\\

$={1\over\eta}{^{*}\sum_{i\in I_{res}}}F_{0,p}({j\over\eta}+{i\over\eta}){2\over\sqrt{\pi t}}{^{*}exp}(-{({i\over\eta})^{2}\over 4t})$\\

Now let ${\eta'}={\eta\over 4}$, ${i'\over \eta'}+{1\over 2\eta'}={i\over\eta}$ and $I'=\{{i-2\over 4}:i\in I\}$. As $F_{0}$ is S-continuous, it is bounded, therefore the same holds for $F_{0,p}$, we also have that ${^{*}exp(-x^{2})}$ is $S$-continuous and rapidly decreasing. Using these properties, we have that;\\

$F_{\kappa\over\nu}({j\over\eta})\simeq {1\over 4\eta'}{^{*}\sum_{i\in I'}}F_{0,p}(({j\over\eta}+{1\over 2\eta'})+{i'\over \eta'}){2\over\sqrt{\pi t}}{^{*}exp}({-({i'\over\eta'}+{1\over 2\eta'})^{2}\over 4t})$\\

$\simeq {1\over\eta'}{^{*}\sum_{i\in I'}}F_{0,p}(({j\over\eta}+{i'\over \eta'}){1\over 2\sqrt{\pi t}}{^{*}exp}(-{({i'\over\eta'})^{2}\over 4t})$\\

Letting $s=\mu i'(i'\in I_{res})$, and $m=|{s-2\over 4}|$, we can, ignoring a finite number of endpoints if necessary, assume that $I'={^{*}\mathcal{Z}}\cap [-m,m]$. This gives the theorem.

\end{proof}

We now verify that our solution defines the classical solution on specialisation, see Lemma \ref{heatequation}, and provides a solution when the initial condition is just $S$-continuous.\\

\begin{theorem}
\label{convolution}
Let $g\in C^{\infty}([0,1])$, see Definition \ref{smooth}, and let $g_{\eta}:\overline{{\Omega}_{\eta}}\rightarrow{^{*}\mathcal{R}}$ be measurable with the definition $g_{\eta}({i\over\eta})={^{*}g}({i\over\eta})$, for $0\leq i\leq \eta-1$. Let $F:\overline{{\Omega}_{\eta}}\times\overline{\mathcal{T}_{\nu}}\rightarrow{^{*}\mathcal{R}}$ satisfy the nonstandard heat equation, as in Lemma \ref{measurable}, with initial condition $g_{\eta}$. Then, for finite $t\in{^{*}\mathcal{R}}$, with ${^{\circ}t}\neq 0$, and $x\in\overline{\Omega_{\eta}}$, we have that;\\

${{^{\circ}}F_{t}(x)}=\int_{\mathcal{R}}g_{per}({^{\circ}x}+y){1\over 2\sqrt{\pi{^{\circ}t}}}e^{{-y^{2}\over 4{^{\circ}t}}}dy$ $(*)$\\

where $g_{per}$ is the periodic extension of $g$ to $\mathcal{R}$. In particular $F$ specialises to the classical standard solution of the heat equation.\\

If $g\in C[0,1]$, see Definition \ref{smooth}, with the same assumptions as in the first part of the Theorem, then $(*)$ still holds.
\end{theorem}
\begin{proof}
By Theorem \ref{solution}, if $\kappa$ is odd, and ${^{\circ}t}\neq 0$, using $S$-integrability of;\\

$F_{0,p}({[x\eta]\over\eta}+y){1\over 2\sqrt{\pi t}}{^{exp}}({-y^{2}\over 4t})$\\

on $\overline{\mathcal{R}_{\eta'}}=\{y\in {^{*}\mathcal{R}}:-m\leq [y\eta']\leq m\}$, equipped with the usual measure $\mu_{m,\eta'}$, and basic facts about specialisation of measures, see \cite{adv}, we have that;\\

${^{\circ}F_{{\kappa\over\nu}}(x)}={^{\circ}}\int_{\overline{\mathcal{R}_{\eta'}}}F_{0,p}({[x\eta]\over\eta}+y){1\over 2\sqrt{\pi t}}{^{*}exp}({-y^{2}\over 4t})d\mu_{m,\eta'}$\\

$=\int_{\mathcal{R}}g_{per}({^{\circ}x}+y){1\over 2\sqrt{\pi{^{\circ}t}}}e^{{-y^{2}\over 4{^{\circ}t}}}dy$\\

By Theorem \ref{nsheat}, we have that ${\partial F\over\partial t}={\partial^{2}F\over \partial^{2}x}$ remains bounded, so that $F_{t}(x)\simeq F_{t+{1\over\nu}}(x)$, for all $x\in\overline{\Omega_{\eta}}$. This gives the first result, as we can assume $\kappa$ is odd. The second claim is well known, see for example \cite{stein}. Finally, if $g\in C[0,1]$, then $F_{0}$ is $S$-continuous and bounded, and so is $F_{{1\over\nu}}$, with $F_{{1\over\nu}}(x)\simeq F_{0}(x)$, for $x\in\overline{\Omega_{\eta}}$. It follows that $F_{{1\over\nu},p}\simeq F_{0,p}$ on $\overline{\mathcal{R}_{\eta'}}$, and $S$-continuous. Taking $F_{{1\over\nu}}$ as the initial condition, we can assume that $\kappa$ is odd. Then, repeating the above proof gives the result.

\end{proof}

For notational reasons, we switch to the interval $[-\pi,\pi]$. The reader is invited to make the relevant transposition to the probability space $[0,1]$ with Lebesgue measure. We make a nonstandard analysis of the heat equation in terms of Fourier series.\\

\begin{defn}
\label{smooth}
We let $S^{1}(1)$ denote the circle of radius $1$, which we identify with the closed interval $[-\pi,\pi]$, via $\mu:[-\pi,\pi]\rightarrow S^{1}(1)$, $\mu(\theta)=e^{i\theta}$. $C(S^{1})$ and $C^{\infty}(S^{1})$ have their conventional meanings. We let $C([-\pi,\pi])=\{\mu^{*}(g):g\in C(S^{1})\}$ and $C^{\infty}([-\pi,\pi])=\{\mu^{*}(g):g\in C^{\infty}(S^{1})\}$. We let $T=[-\pi,\pi]\times\mathcal{R}_{\geq 0}$ and $T^{0}=(-\pi,\pi)\times\mathcal{R}_{>0}$ denote its interior. We let $C(T)=\{G, continuous\ on\ T\ ,G_{t}\in C([-\pi,\pi]), for\ t\in\mathcal{R}_{\geq 0}\}$, $C^{\infty}(T)=\{G\in C(T):G_{t}\in C^{\infty}([-\pi,\pi])$, for $t\in\mathcal{R}_{\geq 0}, G|T^{0}\in C^{\infty}(T^{0})\}$. If $h\in C([-\pi,\pi])$, we define its Fourier coefficient by;\\

$\mathcal{F}(h)(m)={1\over 2\pi}\int_{-\pi}^{\pi}h(x)e^{-imx}dx$\\

for $m\in\mathcal{Z}$. If $g\in C(T)$, we define its Fourier transform in space by;\\

 $\mathcal{F}(g)(m,t)={1\over 2\pi}\int_{-\pi}^{\pi}g(x,t)e^{-imx}dx$\\

 for $m\in\mathcal{Z}$.
\end{defn}

\begin{lemma}
\label{heatequation}
If $g\in C^{\infty}([-\pi,\pi])$, there exists a unique $G\in C^{\infty}(T)$, with $G_{0}=g$, such that $G$ satisfies the heat equation;\\

${\partial{G}\over \partial t}={\partial^{2}G\over \partial x^{2}}$ $(*)$\\

on $T^{0}$.\\

\end{lemma}

\begin{proof}
Suppose, first, there exists such a solution $G$, then, applying $\mathcal{F}$ to $(*)$, we must have that;\\

$\mathcal{F}({\partial{G}\over \partial t}-{\partial^{2}G\over \partial x^{2}})(m,t)=0$ $(t>0,m\in\mathcal{Z})$\\

Differentiating under the integral sign, we have that;\\

 $\mathcal{F}({\partial{G}\over \partial t})={\partial \mathcal{F}(G)\over \partial t}(m,t)$, for $t>0,m\in\mathcal{Z}$\\

 Integrating by parts and using the fact that $G_{t}\in C^{\infty}([-\pi,\pi])$, for $t>0$, we have that;\\

$\mathcal{F}{\partial^{2}G\over \partial x^{2}}=-m^{2}\mathcal{F}(G)(m,t)$, for $t>0,m\in\mathcal{Z}$\\

We thus obtain the sequence of ordinary differential equations, indexed by $m\in\mathcal{Z}$;\\

${\partial \mathcal{F}(G)\over \partial t}+m^{2}\mathcal{F}(G)(m,t)=0$ $(t>0)$\\

As $G\in C(T)$, $G_{t}\rightarrow G_{0}$ pointwise , as $t\rightarrow 0$, and, using the Dominated Convergence Theorem, $\mathcal{F}(G)(m,t)\rightarrow \mathcal{F}(G)(m,0)$, as $t\rightarrow 0$, for each $m\in\mathcal{Z}$. By Picard's and Peano's Theorem, see \cite{adv}, Chapter 4, this system of equations has a unique continuous solution, given by;\\

$\mathcal{F}(G)(m,t)=e^{-m^{2}t}\mathcal{F}(g)(m)$ $(t\geq 0)$\\

As $G_{t}\in C^{\infty}([-\pi,\pi])$, its Fourier series converges absolutely to $G_{t}$ and, in particular, $G_{t}$ is determined by its Fourier coefficients, for $t>0$. It follows that $G$ is a unique solution.\\

If $g\in C^{\infty}([-\pi,\pi])$, its Fourier series converges absolutely to $g$, hence, the series;\\

$\sum_{m\in\mathcal{Z}}e^{-m^{2}t}\mathcal{F}(g)(m)e^{imx}$\\

are absolutely convergent for $t>0$. It follows that $G$ defined by;\\

$G(x,t)=\sum_{m\in\mathcal{Z}}e^{-m^{2}t}\mathcal{F}(g)(m)e^{imx}$\\

is a solution of the required form.
\end{proof}

We introduce more notation.\\

\begin{defn}
If $\eta\in{{^{*}\mathcal{N}}\setminus\mathcal{N}}$, we let $\overline{\mathcal{V}_{\eta}}={^{*}\bigcup_{0\leq i\leq 2\eta-1}}[-\pi+\pi{i\over\eta},-\pi+\pi{i+1\over\eta})$, so that $\overline{\mathcal{V}_{\eta}}={^{*}[-\pi,\pi)}$. We let $\mathcal{D}_{\eta}$ denote the associated $*$-finite algebra, generated by the intervals $[-\pi+\pi{i\over\eta},-\pi+\pi{i+1\over\eta})$, for $0\leq i\leq 2\eta-1$, and $\mu_{\eta}$ the associated counting measure defined by $\mu_{\eta}([-\pi+\pi{i\over\eta},-\pi+\pi{i+1\over\eta}))={\pi\over\eta}$. We let $(\overline{\mathcal{V}_{\eta}},L(\mathcal{D}_{\eta}),L(\mu_{\eta}))$ denote the associated Loeb space, see \cite{L}. If $\nu\in{{^{*}\mathcal{N}}\setminus\mathcal{N}}$, we let $\overline{\mathcal{T}_{\nu}}={^{*}\bigcup_{0\leq i\leq \nu^{2}-1}}[{i\over\nu},{i+1\over\nu})$, so that $\overline{\mathcal{T}_{\nu}}=[0,\nu)\subset{^{*}\mathcal{R}_{\geq 0}}$.We let $\mathcal{C}_{\eta}$ denote the associated $*$-finite algebra, generated by the intervals $[{i\over\nu},{i+1\over\nu})$, for $0\leq i\leq \nu^{2}-1$, and $\lambda_{\nu}$ the associated counting measure defined by $\lambda_{\nu}([{i\over\nu},{i+1\over\nu}))={1\over\nu}$. We let $(\overline{\mathcal{T}_{\nu}},L(\mathcal{C}_{\nu}),L(\lambda_{\nu}))$ denote the associated Loeb space.\\

We let $([-\pi,\pi],\mathfrak{D},\mu)$ denote the interval $[-\pi,\pi]$, with the completion $\mathfrak{D}$ of the Borel field, and $\mu$ the restriction of Lebesgue measure. We let $(\mathcal{R}_{\geq 0}\cup\{+\infty\},\mathfrak{C},\lambda)$ denote the extended real half line, with the completion $\mathfrak{C}$ of the extended Borel field, and $\lambda$ the extension of Lebesgue measure, with $\lambda(+\infty)=\infty$, see \cite{adv}, Chapter 6.\\

We let $(\overline{\mathcal{V}_{\eta}}\times \overline{\mathcal{T}_{\nu}},\mathcal{D}_{\eta}\times\mathcal{C}_{\nu},\mu_{\eta}\times\lambda_{\nu})$ be the associated product space and $(\overline{\mathcal{V}_{\eta}}\times \overline{\mathcal{T}_{\nu}},L(\mathcal{D}_{\eta}\times\mathcal{C}_{\eta}),  L(\mu_{\eta}\times \lambda_{\nu}))$ be the corresponding Loeb space. $(\overline{\mathcal{V}_{\eta}}\times \overline{\mathcal{T}_{\nu}},L(\mathcal{D}_{\eta})\times L(\mathcal{C}_{\nu}),L(\mu_{\eta})\times L(\lambda_{\nu}))$ is the complete product of the Loeb spaces $(\overline{\mathcal{V}_{\eta}},L(\mathcal{D}_{\eta}),L(\mu_{\eta}))$ and $(\overline{\mathcal{T}_{\nu}},L(\mathcal{C}_{\nu}),L(\lambda_{\nu}))$. Similarly, $([-\pi,\pi]\times(\mathcal{R}_{\geq 0}\cup\{+\infty\} ,\mathfrak{D}\times\mathfrak{C},\mu\times\lambda)$ is the complete product of $([-\pi,\pi],\mathfrak{D},\mu)$ and $(\mathcal{R}_{\geq 0}\cup\{+\infty\},\mathfrak{C},\lambda)$.\\

We let $({^{*}\mathcal{R}},{^{*}\mathfrak{E}})$ denote the hyperreals, with the transfer of the Borel field $\mathfrak{C}$ on $\mathcal{R}$. A function $f:(\overline{\mathcal{V}_{\eta}},\mathcal{D}_{\eta})\rightarrow ({^{*}\mathcal{R}},{^{*}\mathfrak{E}})$ is measurable, if $f^{-1}:{^{*}\mathfrak{E}}\rightarrow \mathcal{D}_{\eta}$. The same definition holds for ${\mathcal{T}_{\nu}}$.  Similarly, $f:(\overline{\mathcal{V}_{\eta}}\times\overline{\mathcal{T}_{\nu}},\mathcal{D}_{\eta}\times\mathcal{C}_{\nu})\rightarrow ({^{*}\mathcal{R}},{^{*}\mathfrak{E}})$ is measurable, if $f^{-1}:{^{*}\mathfrak{E}}\rightarrow \mathcal{D}_{\eta}\times\mathfrak{C}_{\nu}$. Observe that this is equivalent to the definition given in \cite{L}. We will abbreviate this notation to $f:\overline{\mathcal{V}_{\eta}}\rightarrow{^{*}\mathcal{R}}$, $f:\overline{\mathcal{V}_{\eta}}\rightarrow{^{*}\mathcal{R}}$ or $f:\overline{\mathcal{V}_{\eta}}\times\overline{\mathcal{T}_{\nu}}\rightarrow{^{*}\mathcal{R}}$ is measurable, $(*)$. The same
applies to $({^{*}\mathcal{C}},{^{*}\mathfrak{E}})$, the hyper complex numbers, with the transfer of the Borel field $\mathfrak{E}$, generated by the complex topology. Observe that $f:\overline{\mathcal{V}_{\eta}}\rightarrow{^{*}\mathcal{C}}$, $f:\overline{\mathcal{T}_{\nu}}\rightarrow{^{*}\mathcal{C}}$  $f:\overline{\mathcal{V}_{\eta}}\times\overline{\mathcal{T}_{\nu}}\rightarrow{^{*}\mathcal{C}}$ is measurable, in this sense, iff $Re(f)$ and $Im(f)$ are measurable in the sense of $(*)$.\\

We let $\overline{\mathcal{S}_{\eta,\nu}}=\overline{\mathcal{V}_{\eta}}\times\overline{\mathcal{T}_{\nu}}$ and;\\

$V(\overline{\mathcal{V}_{\eta}})=\{f:\overline{\mathcal{V}_{\eta}}\rightarrow {^{*}}{\mathcal{C}},\ f\ measurable\ d(\mu_{\eta})\}$\\

and, similarly, we define $V(\overline{\mathcal{T}_{\nu}})$. Let;\\

$V(\overline{\mathcal{S}_{\eta,\nu}})=\{f:\overline{\mathcal{S}_{\eta,\nu}}\rightarrow {^{*}}{\mathcal{C}},\ f\ measurable\ d(\mu_{\eta}\times\lambda_{\nu})\}$\\

\end{defn}

\begin{lemma}
\label{squaremmp}

The identity;\\

$i:(\overline{\mathcal{V}_{\eta}}\times \overline{\mathcal{T}_{\nu}}, L(\mathcal{D}_{\eta}\times \mathcal{C}_{\nu}),L(\mu_{\eta}\times \lambda_{\nu}) )$\\

$\rightarrow (\overline{\mathcal{V}_{\eta}}\times \overline{\mathcal{T}_{\nu}}, L(\mathcal{D}_{\eta})\times L(\mathcal{C}_{\nu}), L(\mu_{\eta})\times L(\lambda_{\nu}))$\\

and the standard part mapping;\\

$st:(\overline{\mathcal{V}_{\eta}}\times \overline{\mathcal{T}_{\nu}}, L(\mathcal{D}_{\eta})\times L(\mathcal{C}_{\nu}), L(\mu_{\eta})\times L(\lambda_{\nu}))\rightarrow [-\pi,\pi]\times \mathcal{R}_{\geq 0}\cup\{+\infty\} $\\

are measurable and measure preserving.
\end{lemma}
\begin{proof}
The proof is similar to work in \cite{adv}, Chapter 6, using Caratheodory's Extension Theorem and Theorem 22 of \cite{and}.\\
\end{proof}

\begin{defn}{Discrete Partial Derivatives}\\
\label{partials}

Let $f:\overline{\mathcal{V}_{\eta}}\rightarrow{^{*}\mathcal{C}}$ be measurable. We define the discrete derivative $f'$ to be the unique measurable function satisfying;\\

$f'(-\pi+\pi{i\over\eta})={\eta\over2\pi}(f(-\pi+\pi{i+1\over\eta})-f(-\pi+\pi{i-1\over\eta}))$;\\

for $i\in{^{*}\mathcal{N}}_{1\leq i\leq 2\eta-2}$.\\

$f'(\pi-{\pi\over\eta})={\eta\over 2\pi}(f(-\pi)-f(\pi-\pi{2\over\eta}))$\\

$f'(-\pi)={\eta\over 2\pi}(f(-\pi+{\pi\over\eta})-f(\pi-{\pi\over\eta}))$\\

Let $f:\overline{\mathcal{T}_{\nu}}\rightarrow{^{*}\mathcal{C}}$ be measurable. We define the discrete derivative $f'$ to be the unique measurable function satisfying;\\

$f'({i\over\nu})=\nu(f({i+1\over\nu})-f({i\over\nu}))$;\\

for $i\in{^{*}\mathcal{N}}_{0\leq i\leq \nu^{2}-2}$.\\

$f'({\nu-1\over\nu})=0$;\\

If $f:\overline{\mathcal{V}_{\eta}}\rightarrow{^{*}\mathcal{C}}$ is measurable, then we define the shift (left, right);\\

$f^{lsh}(-\pi+\pi{j\over \eta})=f(-\pi+\pi{j+1\over\eta})$ for $0\leq j\leq 2\eta-2$\\

$f^{lsh}(\eta-{\pi\over \eta})=f(-\pi)$\\

$f^{rsh}(-\pi+\pi{j\over\eta})=f(-\pi+\pi{j-1\over\eta})$ for $1\leq j\leq 2\eta-1$\\

$f^{rsh}(-\pi)=f(\pi-{\pi\over\eta})$\\

If $f:\overline{\mathcal{T}_{\nu}}\rightarrow{^{*}\mathcal{C}}$ is measurable, then we define the shift (left, right);\\

$f^{lsh}({j\over\nu})=f({j+1\over\nu})$ for $0\leq j\leq \nu^{2}-2$\\

$f^{lsh}(\nu-{1\over\nu})=f(0)$\\

$f^{rsh}({j\over\nu})=f({j-1\over\nu})$ for $1\leq j\leq \nu^{2}-1$\\

$f^{rsh}(0)=f(\nu-{1\over\nu})$\\

If $f:\overline{\mathcal{V}_{\eta}}\times\overline{\mathcal{T}_{\nu}}\rightarrow{^{*}\mathcal{C}}$ is measurable. Then we define $\{{\partial f\over\partial x},{\partial f\over\partial t}\}$ to be the unique measurable functions satisfying;\\

${\partial f\over\partial x}(-\pi+\pi{i\over\eta},t)={\eta\over 2\pi}(f(-\pi+\pi{i+1\over\eta},t)-f(-\pi+\pi{i-1\over\eta},t))$;\\

for $i\in{^{*}\mathcal{N}}_{1\leq i\leq 2\eta-2}, t\in\overline{\mathcal{T}_{\nu}}$\\

${\partial f\over\partial x}(\pi-{\pi\over\eta},t)={\eta\over 2\pi}(f(-\pi,t)-f(\pi-\pi{2\over\eta},t))$\\

${\partial f\over\partial x}(-\pi,t)={\eta\over 2\pi}(f(-\pi+{\pi\over\eta},t)-f(\pi-{\pi\over\eta},t))$\\

${\partial f\over\partial t}(x,{j\over\nu})=\nu(f(x,{j+1\over\nu})-f(x,{j\over\nu}))$;\\

for $j\in{^{*}\mathcal{N}}_{0\leq j\leq \nu^{2}-2}, x\in\overline{\mathcal{H}_{\eta}}$\\

${\partial f\over\partial t}(x,\nu-{1\over\nu})=0$\\

We define $\{f^{lsh_{x}},f^{lsh_{t}},f^{rsh_{x}},f^{rsh_{t}}\}$ by;\\

$f^{lsh_{x}}(x_{0},t_{0})=(f_{t_{0}})^{lsh}(x_{0})$\\

$f^{lsh_{t}}(x_{0},t_{0})=(f_{x_{0}})^{lsh}(t_{0})$\\

$f^{rsh_{x}}(x_{0},t_{0})=(f_{t_{0}})^{rsh}(x_{0})$\\

$f^{rsh_{t}}(x_{0},t_{0})=(f_{x_{0}})^{rsh}(t_{0})$\\

where, if $(x_{0},t_{0})\in\overline{\mathcal{V}_{\eta}}\times\overline{\mathcal{T}_{\nu}}$;\\

$f_{t_{0}}(x_{0})=f_{x_{0}}(t_{0})=f(\pi{[{\eta x_{0}\over \pi}]\over\eta},{[\nu t_{0}]\over \nu})$\\

\end{defn}

\begin{lemma}
\label{rmkpartials}
If $f$ is measurable, then so are;\\

 $\{{\partial f\over\partial x},{\partial f\over\partial t},{\partial^{2} f\over\partial x^{2}},f_{x},f_{t},f^{lsh_{x}},f^{lsh_{t}},f^{rsh_{x}},f^{rsh_{t}},f^{lsh_{x}^{2}},f^{lsh_{t}^{2}},f^{rsh_{x}^{2}},f^{rsh_{t}^{2}}\}$\\

\end{lemma}
\begin{proof}
 This follows immediately, by transfer, from the corresponding result for the discrete derivatives and shifts of discrete functions $f:{\mathcal{H}_{n}}\times{\mathcal{T}_{m}}\rightarrow\mathcal{C}$, where $n,m\in{\mathcal{N}}$, see \cite{adv}, Chapter 6.
\end{proof}

\begin{lemma}
\label{discretederivativeshift}

Let $g,h:\overline{\mathcal{V}_{\eta}}\rightarrow{^{*}\mathcal{C}}$ be measurable. Then;\\

$(i)$. $\int_{\overline{\mathcal{V}_{\eta}}}g'(y)d\mu_{\eta}(y)=0$\\

$(ii)$. $(gh)'=g'h^{lsh}+g^{rsh}h'$\\

$(iii)$. $\int_{\overline{\mathcal{V}_{\eta}}}(g'h)(y)d\mu_{\eta}(y)=-\int_{\overline{\mathcal{V}_{\eta}}}gh'd\mu_{\eta}(y)$\\

$(iv)$. $\int_{\overline{\mathcal{V}_{\eta}}}g(y)d\mu_{\eta}(y)=\int_{\overline{\mathcal{V}_{\eta}}}g^{lsh}(y)d\mu_{\eta}(y)=\int_{\overline{\mathcal{V}_{\eta}}}g^{rsh}(y)d\mu_{\eta}(y)$\\

$(v)$. $(g')^{rsh}=(g^{rsh})'$, $(g')^{lsh}=(g^{lsh})'$ \\

$(vi)$. $\int_{\overline{\mathcal{V}_{\eta}}}(g''h)(y)d\mu_{\eta}(y)=\int_{\overline{\mathcal{V}_{\eta}}}(gh'')(y)d\mu_{\eta}(y)$\\

 \end{lemma}

\begin{proof}
 In the first part, for $(i)$, we have, using Definition \ref{partials}, that;\\

$\int_{\overline{\mathcal{V}_{\eta}}}g'(y)d\mu_{\eta}(y)$\\

$={\pi\over\eta}[{^{*}}\sum_{1\leq j\leq 2\eta-2}{\eta\over 2\pi}[g(-\pi+\pi({j+1\over\eta}))-g(-\pi+\pi({j-1\over\eta}))]$\\

$+{\eta\over 2\pi}[g(-\pi+{\pi\over\eta})-g(\pi-{\pi\over\eta})]+{\eta\over 2\pi}[g(-\pi)-g(\pi-2{\pi\over\eta})]]=0$\\

For $(ii)$, we calculate;\\

$(gh)'(-\pi+\pi{j\over\eta})=$\\

$={\eta\over 2\pi}(gh(-\pi+\pi{j+1\over\eta})-gh(-\pi+\pi{j-1\over\eta}))$\\

$={\eta\over 2\pi}(gh(-\pi+\pi{j+1\over\eta})-g(-\pi+\pi{j-1\over\eta})h(-\pi+\pi{j+1\over\eta})$\\

$+g(-\pi+\pi{j-1\over\eta})h(-\pi+\pi{j+1\over\eta})-gh(-\pi+\pi{j-1\over\eta}))$\\

$=g'(-\pi+\pi{j\over\eta})h(-\pi+\pi{j+1\over\eta})+g(-\pi+\pi{j-1\over\eta})h'(-\pi+\pi{j\over\eta})$\\

$=(g'h^{lsh}+g^{rsh}h')(-\pi+\pi{j\over \eta})$\\

Combining $(i),(ii)$, we have;\\

$0=\int_{\overline{\mathcal{V}_{\eta}}}(gh)'(x)d\mu_{\eta}(x)$\\

$=\int_{\overline{\mathcal{V}_{\eta}}}(g'h^{lsh}+g^{rsh}h')(x)d\mu_{\eta}(x)$\\

and, rearranging, that;\\

$\int_{\overline{\mathcal{V}_{\eta}}}(g'h^{lsh})d\mu_{\eta}=-\int_{\overline{\mathcal{V}_{\eta}}}(g^{rsh}h')d\mu_{\eta}$\\

For $(iv)$, we have that;\\

$\int_{\overline{\mathcal{V}_{\eta}}}g^{rsh}(y)d\mu_{\eta}(y)$\\

$={\pi\over \eta}({^{*}\sum}_{0\leq j\leq 2\eta-1}g^{rsh}(-\pi+\pi{j\over\eta}))$\\

$={\pi\over\eta}({^{*}\sum}_{1\leq j\leq 2\eta-2}g(-\pi+\pi{j-1\over\eta})+g(\pi-{\pi\over\eta}))$\\

$={\pi\over\eta}({^{*}\sum}_{0\leq j\leq 2\eta-1}g(-\pi+\pi{j\over\eta})$\\

$=\int_{\overline{\mathcal{V}_{\eta}}}g(y)d\mu_{\eta}(y)$\\

A similar calculation holds with $g^{lsh}$. For $(v)$, we have for $2\leq j\leq 2\eta-2$;\\

$(g')^{rsh}(-\pi+\pi{j\over\eta})$\\

$=g'(-\pi+\pi{j-1\over \eta})$\\

$={\eta\over 2\pi}(g(-\pi+\pi{j\over \eta})-g(-\pi+\pi{j-2\over \eta}))$\\

$(g^{rsh})'(-\pi+\pi{j\over\eta})$\\

$={\eta\over 2\pi}(g^{rsh}(-\pi+\pi{j+1\over \eta})-g^{rsh}(-\pi+\pi{j-1\over \eta}))$\\

$={\eta\over 2\pi}(g(-\pi+\pi{j\over \eta})-g(-\pi+\pi{j-2\over \eta}))$\\

Similar calculations hold for the remaining $j$ to give that $(g')^{rsh}=(g^{rsh})'$, and the calculation $(g')^{lsh}=(g^{lsh})'$ is also similar.\\

It follows that;\\

$\int_{\overline{\mathcal{V}_{\eta}}}(g'h)d\mu_{\eta}$\\

$=\int_{\overline{\mathcal{V}_{\eta}}}(g'(h^{rsh})^{lsh})d\mu_{\eta}$\\

$=-\int_{\overline{\mathcal{V}_{\eta}}}(g^{rsh}(h^{rsh})')d\mu_{\eta}$\\

$=-\int_{\overline{\mathcal{V}_{\eta}}}(g^{rsh}(h')^{rsh})d\mu_{\eta}$\\

$=-\int_{\overline{\mathcal{V}_{\eta}}}(gh'))d\mu_{\eta}$\\

which gives $(iii)$, using $(iv),(v)$. The calculation $(vi)$ is then immediate from $(iii)$.\\

\end{proof}

\begin{lemma}
\label{simpartials}
Similar results to Lemma \ref{discretederivativeshift} hold  for $\{lsh_{x},rsh_{x},{\partial \over \partial x},{\partial \over \partial t}\}$. Namely, if $g,h:\overline{\mathcal{S}_{\eta,\nu}}\rightarrow{^{*}\mathcal{C}}$ are measurable. Then;\\

$(i)$. $\int_{\overline{\mathcal{S}_{\eta,\nu}}}{\partial g\over \partial x}d(\mu_{\eta}\times\lambda_{\nu})=0$\\

$(ii)$. ${\partial gh\over \partial x}={\partial g\over \partial x}h^{lsh_{x}}+g^{rsh_{x}}{\partial h\over \partial x}$\\

$(iii)$. $\int_{\overline{\mathcal{S}_{\eta,\nu}}}{\partial g\over \partial x}hd(\mu_{\eta}\times\lambda_{\nu})=-\int_{\overline{\mathcal{S}_{\eta,\nu}}}g{\partial h\over \partial x}d(\mu_{\eta}\times\lambda_{\nu})$\\

$(iv)$. $\int_{\overline{\mathcal{S}_{\eta,\nu}}}gd(\mu_{\eta}\times\lambda_{\nu})=\int_{\overline{\mathcal{S}_{\eta,\nu}}}g^{lsh_{x}}d(\mu_{\eta}\times\lambda_{\nu})=\int_{\overline{\mathcal{S}_{\eta,\nu}}}g^{rsh_{x}}d(\mu_{\eta}\times\lambda_{\nu})$\\

$(v)$. $({\partial g\over \partial x})^{lsh_{x}}={\partial (g^{lsh_{x}})\over \partial x}$, and, similarly, with $rsh_{x}$ replacing $lsh_{x}$. \\

$(vi)$. $\int_{\overline{\mathcal{S}_{\eta,\nu}}}({\partial^{2} g\over \partial x^{2}}h)d(\mu_{\eta}\times\lambda_{\nu})=\int_{\overline{\mathcal{S}_{\eta,\nu}}}(g{\partial^{2} h\over \partial x^{2}})d(\mu_{\eta}\times\lambda_{\nu})$ $(*)$\\

\end{lemma}

\begin{proof}

For $(i)$, using $(i)$ from the argument in Lemma \ref{discretederivativeshift}, we have;\\

$\int_{\overline{\mathcal{S}_{\eta,\nu}}}{\partial g\over \partial x}d(\mu_{\eta}\times\lambda_{\nu})$\\

$=\int_{\overline{\mathcal{V}_{\eta}}}(\int_{\overline{\mathcal{T}_{\nu}}}({\partial g\over \partial x})_{t}d\mu_{\eta})d\lambda_{\nu}(t)$\\

$=\int_{\overline{\mathcal{V}_{\eta}}}(\int_{\overline{\mathcal{T}_{\nu}}}({\partial g_{t}\over \partial x})d\mu_{\eta})d\lambda_{\nu}(t)$\\

$=\int_{\overline{\mathcal{T}_{\nu}}}0d\lambda_{\nu}(t)=0$\\

The proofs of $(ii),(iii),(iv)$ are similar to Lemma \ref{discretederivativeshift}, relying on the result of $(i)$. $(v)$ follows easily from Definitions \ref{partials} and $(vi)$ follows, repeating the result of $(iii)$, and applying $(v)$.\\

\end{proof}

\begin{defn}
\label{restriction}
If $\eta$ is even, we define a restriction $\overline{()}:\overline{\mathcal{V}_{\eta}}\rightarrow \overline{\mathcal{V}_{{\eta\over 2}}}$. Namely;\\

$\overline{f}(-\pi+\pi{2i\over\eta})=f(-\pi+\pi{2i\over\eta})$;\\

for $i\in{^{*}\mathcal{N}}_{0\leq i\leq \eta-1}$.\\
\end{defn}

\begin{lemma}
\label{restrictionderivative}
Let notation be as in Definitions \ref{restriction} and \ref{partials}, then;\\

$\overline{f'}(-\pi+\pi{2 i\over\eta})={\eta\over2\pi}(f(-\pi+\pi{2 i+1\over\eta})-f(-\pi+\pi{2 i-1\over\eta}))$;\\

for $i\in{^{*}\mathcal{N}}_{1\leq i\leq \eta-1}$.\\

$\overline{f'}(-\pi)={\eta\over 2\pi}(f(-\pi+{\pi\over\eta})-f(\pi-{\pi\over\eta}))$\\

and;\\

$\overline{f^{lsh}}(-\pi+\pi{2 j\over \eta})=f(-\pi+\pi{2j+1\over\eta})$ for $0\leq j\leq \eta-1$\\

$\overline{f^{rsh}}(-\pi+\pi{2 j\over\eta})=f(-\pi+\pi{2 j-1\over\eta})$ for $1\leq j\leq \eta-1$\\

$\overline{f^{rsh}}(-\pi)=f(\pi-{\pi\over\eta})$\\

\end{lemma}

\begin{proof}
The proof is an immediate consequence of Definitions \ref{restriction} and \ref{partials}

\end{proof}

\begin{rmk}
\label{nequal}
It is important to note that, in general $\overline{f'}\neq \overline{f}'$ and, similarly, for $lsh, rsh$.

\end{rmk}

\begin{lemma}
\label{restrictionfacts}
Let $\{g,h\}\subset V(\overline{\mathcal{V}_{\eta}})$ be measurable, then;\\

$(i)$. $\int_{\overline{\mathcal{V}_{{\eta\over 2}}}}\overline{g'}(y)d\mu_{{\eta\over 2}}(y)=0$\\

$(ii)$. $\overline{(gh)'}=\overline{g'}\overline{h^{lsh}}+\overline{g^{rsh}}\overline{h'}$\\

$(iii)$. $\int_{\overline{\mathcal{V}_{{\eta\over 2}}}}\overline{(g'h)}(y)d\mu_{{\eta\over 2}}(y)=-\int_{\overline{\mathcal{V}_{{\eta\over 2}}}}\overline{g^{rsh}(h')^{rsh}}d\mu_{\eta}(y)$\\

$(iv)$. $\int_{\overline{\mathcal{V}_{{\eta\over 2}}}}\overline{g^{rsh^{2}}}(y)d\mu_{{\eta\over 2}}(y)=\int_{\overline{\mathcal{V}_{{\eta\over 2}}}}\overline{g}(y)d\mu_{\eta}(y)$\\

\end{lemma}
\begin{proof}
For $(i)$, we have that;\\

 $\int_{\overline{\mathcal{V}_{{\eta\over 2}}}}\overline{g'}(y)d\mu_{{\eta\over 2}}(y)$\\

$={2\pi\over\eta}[{^{*}}\sum_{1\leq j\leq \eta-1}{\eta\over 2\pi}[g(-\pi+\pi({2j+1\over\eta}))-g(-\pi+\pi({2j-1\over\eta}))]$\\

$+{\eta\over 2\pi}[g(-\pi+{\pi\over\eta})-g(\pi-{\pi\over\eta})]]=0$\\

$(ii)$ is clear from the main proof and taking restrictions.\\

For $(iii)$, integrating both sides of $(ii)$ and using $(i)$, we have that;\\

$\int_{\overline{\mathcal{V}_{{\eta\over 2}}}}\overline{g'h^{lsh}}d\mu_{{\eta\over 2}}(y)=-\int_{\overline{\mathcal{V}_{{\eta\over 2}}}}\overline{g^{rsh}h'}d\mu_{{\eta\over 2}}(y)$ $(*)$\\

Then;\\

$\int_{\overline{\mathcal{V}_{{\eta\over 2}}}}\overline{g'h}d\mu_{{\eta\over 2}}(y)$\\

$=\int_{\overline{\mathcal{V}_{{\eta\over 2}}}}\overline{g'(h^{rsh})^{lsh}}d\mu_{{\eta\over 2}}(y)$\\

$=-\int_{\overline{\mathcal{V}_{{\eta\over 2}}}}\overline{g^{rsh}(h^{rsh})'}d\mu_{{\eta\over 2}}(y)$ by $(*)$\\

$=-\int_{\overline{\mathcal{V}_{{\eta\over 2}}}}\overline{g^{rsh}(h')^{rsh}}d\mu_{{\eta\over 2}}(y)$ by the main proof\\

$(iv)$ is a simple calculation, using Definitions \ref{restriction} and \ref{partials}.

\end{proof}

\begin{lemma}
\label{nsheat}
Given a measurable boundary conditions $f\in V(\overline{\mathcal{V}_{\eta}})$, there exists a unique measurable $F\in V(\overline{\mathcal{S}_{\eta,\nu}})$, satisfying the nonstandard heat equation;\\

${\partial F\over\partial t}={\partial^{2}F\over\partial x^{2}}$\\

on $({\overline{\mathcal{T}_{\nu}}\setminus [\nu-{1\over\nu},\nu)})\times\overline{\mathcal{V}_{\eta}}$\\

with $F(0,x)=f(x)$, for $x\in\overline{\mathcal{V}_{\eta}}$, $(*)$.\\

Moreover, if $\eta\leq\sqrt{2\nu}\pi$, and, there exists $M\in\mathcal{R}$, with $max\{f,f',f''\}\leq M$, then $max\{F,{\partial F\over\partial x},{\partial^{2} F\over\partial x^{2}}\}\leq M$.

\end{lemma}

\begin{proof}
Observe that, by Definition \ref{partials}, if $F:\overline{\mathcal{S}_{\eta,\nu}}\rightarrow{^{*}\mathcal{C}}$ is measurable, then;\\

${\partial^{2}F\over\partial x^{2}}(-\pi+\pi{i\over\eta},t)={\eta^{2}\over 4\pi^{2}}(F(-\pi+\pi{i+2\over\eta},t)-2F(-\pi+\pi{i\over\eta},t)+F(-\pi+\pi{i-2\over\eta},t))$\\

$(2\leq i\leq 2\eta-3), t\in\overline{\mathcal{T}_{\nu}}$, with similar results for the remaining $i$.\\

Therefore, if $F$ satisfies $(*)$, we must have;\\

$F(0,x)=f(x)$, $(x\in\overline{\mathcal{V}_{\eta}})$\\

$F({i+1\over\nu},-\pi+\pi{j\over\eta})$\\

$=F({i\over\nu},-\pi+\pi{j\over\eta})+{\eta^{2}\over 4\pi^{2}\nu}(F({i\over\nu},-\pi+\pi{j+2\over\eta})-2F({i\over\nu},-\pi+\pi{j\over\eta})+F({i\over\nu},-\pi+\pi{j-2\over\eta}))$\\

$={\eta^{2}\over 4\pi^{2}\nu}F({i\over\nu},-\pi+\pi{j+2\over\eta})+(1-{\eta^{2}\over 2\pi^{2}\nu})(F({i\over\nu},-\pi+\pi{j\over\eta})+{\eta^{2}\over 4\pi^{2}\nu}F({i\over\nu},-\pi+\pi{j-2\over\eta}))$ $(*)$\\

$(1\leq i\leq \nu^{2}-2,0\leq j\leq 2\eta-1)$\\

The choice of $\eta$ ensures that $1-{\eta^{2}\over 2\pi^{2}\nu}\geq 0$. Hence, inductively, if $|F_{i\over\nu}|\leq M$, then, by $(*)$;\\

$|F_{i+1\over\nu}|\leq M({\eta^{2}\over 4\pi^{2}\nu}+(1-{\eta^{2}\over 2\pi^{2}\nu})+{\eta^{2}\over 4\pi^{2}\nu})=M$.\\

We can differentiate $(*)$ and replace $F$ with ${\partial F\over\partial x}$ or ${\partial^{2} F\over\partial x^{2}}$. The same argument, and the assumption on the initial conditions, gives the required bound.

\end{proof}

\begin{lemma}
\label{initialbounded}
If $f\in C^{\infty}[-\pi,\pi]$, and $f_{\eta}$ is defined on $\overline{\mathcal{V}_{\eta}}$ by;\\

$f_{\eta}(-\pi+\pi{j\over\eta})=f^{*}(-\pi+\pi{j\over\eta})$\\

$f_{\eta}(x)=f(-\pi+{\pi\over\eta}[{\eta(x+\pi)\over \pi}])$\\

where $f^{*}$ is the transfer of $f$ to ${^{*}[-\pi,\pi)}$, then there exists a constant $M\in\mathcal{R}$, such that $max\{f_{\eta},f_{\eta}',f_{\eta}''\}\leq M$. In particular, if $F$ solves the nonstandard heat equation, with initial condition $f_{\eta}$, then, $max\{F,{\partial F\over\partial x},{\partial^{2} F\over\partial x^{2}}\}\leq M$ as well.

\end{lemma}

\begin{proof}
We have, for $x\in [-\pi,\pi)$, using Taylor's Theorem, that;\\

$|{1\over 2h}(f(x+h)-f(x-h))-f'(x)|$\\

$=|{1\over 2h}(f(x)+hf'(x)+{h^{2}\over 2}f''(c)-f(x)+hf'(x)-{h^{2}\over 2}f''(c'))-f'(x)|$\\

$\leq hK$\\

where $K=max_{[-\pi,\pi)}f''$. By transfer, it follows, that, for infinite $\eta$, $(f_{\eta})'\simeq (f')_{\eta}$. Clearly $(f')_{\eta}$ is bounded, as $f'$ is, which gives the result for $(f_{\eta})'$. The case for $(f_{\eta})''$ is similar. The final result is immediate from Lemma \ref{nsheat}.
\end{proof}

\begin{defn}
\label{nstransf}
We let $\overline{\mathcal{Z}_{\eta}}=\{m\in{{^{*}}\mathcal{Z}}:-\eta\leq m\leq \eta-1\}$ Given a measurable $f:\overline{\mathcal{V}_{\eta}}\rightarrow{^{*}\mathcal{C}}$, we define, for $m\in\mathcal{Z}_{\eta}$, the $m'th$ discrete Fourier coefficient to be;\\

$\hat{f}_{\eta}(m)={1\over 2\pi}\int_{\overline{\mathcal{V}_{\eta}}}f(y)exp_{\eta}(-iym)d\mu_{\eta}(y)$\\

\end{defn}

\begin{lemma}
\label{fourinversion}

Let hypotheses be as in Definition \ref{nstransf}, then;\\

$f(x)=\sum_{m\in\mathcal{Z}_{\eta}}\hat{f}_{\eta}(m)exp_{\eta}(ixm)$

\end{lemma}

\begin{proof}

This is a simple transposition of Lemma 5.9 in \cite{adv}, Chapter 5. We have there that the measure on $\overline{\mathcal{S}_{\eta}}$ is $\lambda_{\eta}$. The result follows using the scalar map $p:\overline{\mathcal{V}_{\eta}}\rightarrow\overline{\mathcal{S}_{\eta}}$, $p(x)={x\over\pi}$, and the fact that $p_{*}(\mu_{\eta})=\lambda_{\eta}$\\

\end{proof}

\begin{defn}
\label{verthoriz}

Given a measurable $f:\overline{\mathcal{S}_{\eta,\nu}}\rightarrow{^{*}\mathcal{C}}$, we define the nonstandard vertical Fourier transform $\hat{f}:\overline{\mathcal{T}_{\nu}}\times\overline{\mathcal{Z}_{\eta}}\rightarrow{^{*}\mathcal{C}}$ by;\\

$\hat{f}(t,m)={1\over 2\pi}\int_{\overline{\mathcal{V}_{\eta}}}f(t,x)exp_{\eta}(-ixm)d\mu_{\eta}(x)$\\

and, given a measurable $g:\overline{\mathcal{T}_{\nu}}\times\overline{\mathcal{Z}_{\eta}}\rightarrow{^{*}\mathcal{C}}$, we define the nonstandard inverse vertical Fourier transform by;\\

$\check{g}(t,x)=\sum_{m\in\mathcal{Z}_{\eta}}g(t,m)exp_{\eta}(ixm)$\\

so that, by Lemma \ref{fourinversion}, $f=\check{\hat{f}}$\\

Similar to Definition 6.20 of \cite{adv}, Chapter 6, for $f\in\overline{\mathcal{V_{\eta}}}$, we let $\phi_{\eta}:\overline{\mathcal{Z}_{\eta}}\rightarrow{^{*}\mathcal{C}}$ be defined by;\\

$\phi_{\eta}(m)={\eta\over 2\pi}(exp_{\eta}(-im{\pi\over\eta})-exp_{\eta}(im{\pi\over\eta}))$\\

We let $\psi_{\eta}:\overline{\mathcal{Z}_{{\eta\over 2}}}\rightarrow{^{*}\mathcal{C}}$ be defined by;\\

$\psi_{\eta}(m)={\eta\over 2\pi}(1-exp_{\eta}(im{2\pi\over\eta}))$\\

and, we let $U_{\eta}:\overline{\mathcal{Z}_{{\eta\over 2}}}\rightarrow{^{*}\mathcal{C}}$ be defined by;\\

$U_{\eta}(m)=exp_{\eta}(-im{2\pi\over\eta}))$\\

\end{defn}

The following is the analogue of Lemma 5.14 in \cite{adv}, Chapter 5, using the definition of the discrete derivative in Definition \ref{partials} and the discrete Fourier coefficients from Definition \ref{nstransf};\\

\begin{lemma}
\label{sine}
Let $f:\overline{\mathcal{V}_{\eta}}\rightarrow{^{*}\mathcal{C}}$ be measurable; then, for $m\in\mathcal{Z}_{\eta}$,\\

$\hat{f''}(m)=\phi_{\eta}^{2}(m)\hat{f}(m)$\\

\end{lemma}

\begin{proof}

We have, using Lemma \ref{discretederivativeshift}(iii), that;\\

$(\hat{f'})(m)={1\over 2\pi}\int_{\overline{\mathcal{V}_{\eta}}}f'(x)exp_{\eta}(-ixm)d\mu_{\eta}(x)$\\

$=-{1\over 2\pi}\int_{\overline{\mathcal{V}_{\eta}}}f(x)(exp_{\eta})'(-ixm)d\mu_{\eta}(x)$\\

A simple calculation shows that;\\

$(exp_{\eta})'(-ixm)=exp_{\eta}(-ixm)\phi_{\eta}(m)$\\

Therefore;\\

$(\hat{f'})(m)=-\phi_{\eta}(m)\hat{f}(m)$\\

Then $\hat{f''}(m)$\\

$=-\phi_{\eta}(m)\hat{f'}(m)$\\

$=\phi_{\eta}^{2}(m)\hat{f}(m)$\\

as required.
\end{proof}

\begin{lemma}
\label{restrictedsine}
If $f:\overline{\mathcal{V}_{\eta}}\rightarrow{^{*}\mathcal{C}}$ is measurable, then, for $m\in\mathcal{Z}_{{\eta\over 2}}$, we have that;\\

$\hat{\overline{f''}}(m)=\psi_{\eta}(m)^{2}U_{\eta}(m)(\hat{\overline{f}}(m))$\\

\end{lemma}

\begin{proof}
 we have, using Lemma \ref{restrictionfacts}(iii), that;\\

$\hat{\overline{f'}}(m)={1\over 2\pi}\int_{\overline{\mathcal{V}_{{\eta\over 2}}}}\overline{f'}(x)exp_{{\eta\over 2}}(-ixm)d\mu_{{\eta\over 2}}(x)$\\

$={1\over 2\pi}\int_{\overline{\mathcal{V}_{{\eta\over 2}}}}\overline{f'}(x)\overline{exp_{\eta}}(-ixm)d\mu_{{\eta\over 2}}(x)$\\

$=-{1\over 2\pi}\int_{\overline{\mathcal{V}_{{\eta\over 2}}}}\overline{f^{rsh}}(x)\overline{{exp'_{\eta}}^{rsh}}(-ixm)d\mu_{{\eta\over 2}}(x)$\\

We calculate;\\

$\overline{{exp'_{\eta}}^{rsh}}(-im(-\pi+\pi{2j\over\eta}))$\\

$={exp'_{\eta}}(-im(-\pi+\pi{2j-1\over\eta}))$\\

$={\eta\over 2\pi}({exp_{\eta}}(-im(-\pi+\pi{2j\over\eta}))-{exp_{\eta}}(-im(-\pi+\pi{2j-2\over\eta})))$\\

$={exp_{\eta}}(-im(-\pi+\pi{2j\over\eta}))[{\eta\over 2\pi}(1-exp_{\eta}(im({2\pi\over\eta})))]$\\

$=\psi_{\eta}(m)\overline{exp_{\eta}}(-im(-\pi+\pi{2j\over\eta}))$\\

Then;\\

$\hat{\overline{f'}}(m)=-{1\over 2\pi}\psi_{\eta}(m)\int_{\overline{\mathcal{V}_{{\eta\over 2}}}}\overline{f^{rsh}}(x)\overline{{exp_{\eta}}}(-ixm)d\mu_{{\eta\over 2}}(x)$\\

$=-\psi_{\eta}(m)(\hat{\overline{f^{rsh}}}(m))$\\

It follows that;\\

$\hat{\overline{f''}}(m)=-\psi_{\eta}(m)(\hat{\overline{(f')^{rsh}}}(m))$\\

$=-\psi_{\eta}(m)(\hat{\overline{(f^{rsh})'}}(m))$\\

$=\psi_{\eta}(m)^{2}(\hat{\overline{f^{rsh^{2}}}}(m))$\\

We calculate, using Lemma \ref{restrictionfacts}(iv);\\

$\hat{\overline{f^{rsh^{2}}}}(m)$\\

$={1\over 2\pi}\int_{\overline{\mathcal{V}_{{\eta\over 2}}}}\overline{f^{rsh^{2}}}(x)\overline{exp_{\eta}}(-ixm)d\mu_{{\eta\over 2}}(x)$\\

$={1\over 2\pi}exp_{\eta}({-2\pi im\over\eta})\int_{\overline{\mathcal{V}_{{\eta\over 2}}}}\overline{f^{rsh^{2}}}(x)\overline{exp_{\eta}^{rsh^{2}}}(-ixm)d\mu_{{\eta\over 2}}(x)$\\

$={1\over 2\pi}U_{\eta}(m)\int_{\overline{\mathcal{V}_{{\eta\over 2}}}}\overline{f}(x)\overline{exp_{\eta}}(-ixm)d\mu_{{\eta\over 2}}(x)$\\

$=U_{\eta}(m)\hat{\overline{f}}(m)$\\

Hence;\\

$\hat{\overline{f''}}(m)=\psi_{\eta}(m)^{2}U_{\eta}(m)(\hat{\overline{f}}(m))$\\

as required.

\end{proof}

\begin{lemma}
\label{decayrate}

If $f\in V(\overline{\mathcal{V}_{\eta}})$, with $f''$ bounded, then, there exists a constant $F\in\mathcal{R}$, with;\\

$|\hat{\overline{f}}(m)|\leq {F\over m^{2}}$, for $m\in\mathcal{Z}_{{\eta\over 2}}$.\\

Moreover;\\

$({^{\circ}\overline{f}})(x)=\sum_{m\in\mathcal{Z}}{^{\circ}((\hat{\overline{f}})(m))}exp(im{^{\circ}x})$, $x\in\overline{\mathcal{V}_{\eta}}$.\\

\end{lemma}

\begin{proof}
Using results of \cite{adv}, Chapter 5, we have that ${2|m|\over\pi}\leq |\psi_{\eta}(m)|\leq {4|m|\over\pi}$, and $|U_{\eta}(m)|=1$ for $|m|\leq{\eta\over 2}$. As $f''$ is bounded, so is $\overline{f''}$, so  $|\hat{\overline{f''}}(m)|\leq G\in\mathcal{R}$. This implies, by the result of Lemma \ref{restrictedsine} that;\\

$|\hat{\overline{f}}(m)|\leq {F\over m^{2}}$. $(*)$\\

for $m\in\mathcal{Z}_{\eta\over 2}$, where $F={G\pi^{2}\over 4}$, as required. Using the Inversion Theorem from Lemma \ref{fourinversion}, we have that;\\

$\overline{f}(x)={^{*}\sum}_{m\in\mathcal{Z}_{\eta\over 2}}\hat{\overline{f}}(m)exp_{{\eta\over 2}}(ixm)=M$, $(x\in\overline{\mathcal{V}_{\eta}})$\\

Let $L=\sum_{m\in\mathcal{Z}}{^{\circ}(\hat{\overline{f}}(m))}exp(im{^{\circ}x})$\\

If $\epsilon>0$, we have, using these results, and the fact that $exp_{{\eta\over 2}}$ is $S$-continuous, that for $n\in\mathcal{Z}$;\\

$|M-L|\leq |M-M_{n}|+|M_{n}-L_{n}|+|L-L_{n}|$\\

$\leq {^{*}\sum}_{n+1\leq |m|\leq {\eta\over 2}}{F\over m^{2}}+\sum_{m=1}^{n}\delta_{i}+\sum_{|m|\geq n+1}{F+1\over m^{2}}$ $(\delta_{i}\simeq 0)$\\

$\leq {2F({1\over n}-{2\over \eta})}+\delta+{2(F+1)\over n}$ $(\delta\simeq 0)$\\

$\leq {4(F+1)\over n}<\epsilon$\\

for $n>{4(F+1)\over \epsilon}$, $n\in\mathcal{N}$. As $\epsilon$ was arbitrary, we obtain the result.

\end{proof}

\begin{lemma}
\label{coefficient}
If $F$ solves the nonstandard heat equation, with initial condition $f$, bounded and $S$-continuous, such that ${^{\circ}f}(x)=g({^{\circ}x})$, where $g$ is continuous and bounded on $[-\pi,\pi]$, then;\\

${^{\circ}(\hat{\overline{F}}(m,t))}=e^{-m^{2}{^{\circ}t}}(\hat{g})(m)$\\

for $m\in\mathcal{Z}$ and finite $t$.\\

\end{lemma}

\begin{proof}
As ${\partial F\over \partial t}-{\partial^{2}F\over \partial x^{2}}=0$, we have, taking restrictions, that;\\

$\overline{{\partial F\over \partial t}}-\overline{{\partial^{2}F\over \partial x^{2}}}=0$\\

Taking Fourier coefficients, for $m\in\mathcal{Z}$, and, using Lemma \ref{restrictedsine};\\

${d\hat{\overline{F}}(m,t)\over dt}-\theta_{\eta}(m)\hat{\overline{F}}(m,t)=0$\\

where $\theta_{\eta}(m)=\psi_{\eta}^{2}(m)U_{\eta}(m)$. Then;\\

$\nu(\hat{\overline{F}}(m,t+{1\over \nu})-\hat{\overline{F}(m,t)})=\theta_{\eta}(m)\hat{\overline{F}}(m,t)$\\

Rearranging, we obtain;\\

$\hat{\overline{F}}(m,t+{1\over\nu})=(1+{\theta_{\eta}(m)\over \nu})\hat{\overline{F}}(m,t)$\\

and, solving the recurrence;\\

$\hat{\overline{F}}(m,t)=(1+{\theta_{\eta}(m)\over \nu})^{[\nu t]}\hat{\overline{F}}(m,0)$\\

Taking standard parts, and using the facts that $lim_{n\rightarrow\infty}(1+{x\over n})^{n}=e^{x}$, and ${^{\circ}{\theta_{\eta}(m)}}=-m^{2}$, we obtain;\\

${^{\circ}(\hat{\overline{F}}(m,t))}=(e^{-m^{2}{^{\circ}t}}){^{\circ}}(\hat{\overline{f}}(m))$\\

for finite $t$. As $f$ is bounded and $S$-continuous, so is $\overline{f}$, and ${^{\circ}f}={^{\circ}\overline{f}}$ is integrable. We have that;\\

${^{\circ}\int_{\overline{\mathcal{V}_{\eta\over 2}}}\overline{f}(x)exp_{\eta\over 2}(-imx)}d\mu_{\eta\over 2}=\int_{-\pi}^{\pi}{^{\circ}f({^{\circ} x})}e^{-ixm}dx=\int_{-\pi}^{\pi}g(x)e^{-ixm}dx$\\

Hence;\\

${^{\circ}(\hat{\overline{F}}(m,t))}=(e^{-m^{2}{^{\circ}t}}\hat{g}(m))$\\

as required.\\

\end{proof}

\begin{theorem}
\label{specialises}
Let $g\in C^{\infty}([-\pi,\pi])$, and $G$ be as in Lemma \ref{heatequation}. Let $f=g_{\eta}$, and let $F$ be as in Lemma \ref{nsheat}. Then, for finite $t$, and $(x,t)\in\overline{\mathcal{V}_{\eta}}\times\overline{ \mathcal{T}_{\nu}}$, ${^{\circ}F}(x,t)=G({^{\circ}x},{^{\circ}t})$

\end{theorem}

\begin{proof}
By Lemma \ref{initialbounded}, we have that ${\partial^{2}F\over\partial x^{2}}$ is bounded. By Lemma \ref{decayrate};\\

${^{\circ}\overline{F}(x,t)}=\sum_{m\in\mathcal{Z}}{^{\circ}(\hat{\overline{F}}(m,t))}exp(im{^{\circ}x})$ $(*)$\\

By Lemma \ref{coefficient};\\

${^{\circ}(\hat{\overline{F}}(m,t))}=e^{-m^{2}{^{\circ}t}}\mathcal{F}(g)(m)$ $(**)$\\

Comparing $(*),(**)$, with the expression;\\

$G({^{\circ} x},{^{\circ}t})=\sum_{m\in\mathcal{Z}}e^{-m^{2}{^{\circ}t}}\mathcal{F}(g)(m)e^{im{^{\circ}x}}$\\

obtained in Lemma \ref{heatequation}, gives the result that ${^{\circ}\overline{F}(x,t)}=G({^{\circ}x},{^{\circ}t})$. However, $\overline{F}_{t}$ is $S$-continuous, for finite $t$, by the fact that $({\partial F\over \partial x})_{t}$ is bounded, from Lemma \ref{initialbounded}, hence;\\

${^{\circ}\overline{F}(x,t)}={^{\circ}F(x,t)}=G({^{\circ}x},{^{\circ}t})$\\

as required.\\

\end{proof}

\begin{theorem}
\label{equilibrium}
Let $g\in C^{\infty}([-\pi,\pi])$, and $G$ be as in Lemma \ref{heatequation}. Let $f=g_{\eta}$, and let $F$ be as in Lemma \ref{nsheat}. Then, for infinite $t$, and $(x,t)\in\overline{\mathcal{V}_{\eta}}\times\overline{ \mathcal{T}_{\nu}}$, $F(x,t)\simeq \int_{\overline{\mathcal{V}_{\eta}}} f d\mu_{\eta}$.

\end{theorem}

\begin{proof}
Again, using Lemma \ref{initialbounded} and Lemma \ref{decayrate}, we have, using the proof of Lemma \ref{coefficient}, that;\\

$\overline{F}(x,t)\simeq \sum_{m\in\mathcal{Z}_{{\eta\over 2}}}exp_{\nu}^{-\theta_{\eta}(m)t}\hat{f}(m)exp_{{\eta\over 2}}(imx)$\\

Taking standard parts, using Lemma \ref{decayrate}, and the fact that $exp_{\nu}^{-\theta_{\eta}(m)t}\simeq 0$, for finite $m$ and infinite $t$, we see that all the coefficients vanish, except when $m=0$, that is;\\

$\overline{F}(x,t)\simeq \hat{f}(0)=\int_{\overline{\mathcal{V}_{{\eta\over 2}}}}\overline{f}d\mu_{{\eta\over 2}}\simeq \int_{\overline{\mathcal{V}_{\eta}}}f d\mu_{\eta} $\\

The result follows for $F(x,t)$ by $S$-continuity.\\

\end{proof}

\begin{rmk}
When $\eta^{2}=2\pi^{2}\nu$, we obtain, by Lemma \ref{nsheat}, the iterative scheme for the nonstandard Markov chain with transition probabilities $\{{1\over 2},{1\over 2}\}$. By Theorem \ref{equilibrium}, we obtain convergence to equilibrium after at least $\nu^{2}={\eta^{4}\over\pi^{4}}$ steps, which is polynomial in the number of states $\eta$. This is a considerable improvement over the result  in Theorem \ref{measurable}, which is exponential in $\eta$. The discrepancy results from the choice of a "smooth" initial distribution. The method of reverse martingales is useful to consider other "nonsmooth" cases, when the initial condition is just S-continuous, for which a Fourier analysis is impossible.

\end{rmk}

\end{document}